\documentclass[final,12pt]{colt2021} 

\usepackage[utf8]{inputenc} 
\usepackage[T1]{fontenc}    
\usepackage{hyperref}       
\usepackage{url}            
\usepackage{booktabs}       
\usepackage{amsfonts}       
\usepackage{nicefrac}       
\usepackage{microtype}      
\usepackage{amsmath,dsfont,empheq}
\usepackage{algorithm}
\usepackage{algpseudocode}
\usepackage{syntonly,bm,wrapfig}
\usepackage{xcolor,cases,balance}
\input{mysymbol.sty}
\usepackage{makecell}
\newcommand{\EE}{\mathbb{E}}
\DeclarePairedDelimiter{\dotp}{\langle}{\rangle}

\def \bbeta {{\boldsymbol \beta}}

\def \diag {\rm diag}
\newcommand{\blue}[1]{{\color{black}#1}}

\newcommand\numberthis{\addtocounter{equation}{1}\tag{\theequation}}

\newtheorem{theorem}{Theorem}
\newtheorem{lemma}{Lemma}

\makeatletter
\def\thanks#1{\protected@xdef\@thanks{\@thanks
        \protect\footnotetext{#1}}}
\makeatother
\algrenewcommand{\algorithmiccomment}[1]{\hskip0em$\triangleright$ #1}

\newcommand{\FullTitle}{Solving Stochastic Compositional Optimization is Nearly as Easy as Solving Stochastic Optimization}

\title[\FullTitle]{\FullTitle}

\usepackage{times}

\coltauthor{%
\Name{Tianyi Chen} \Email{chent18@rpi.edu}\\
  \addr Rensselaer Polytechnic Institute\\
  \Name{Yuejiao Sun} \Email{sunyj@math.ucla.edu}\\
  \addr   University of California, Los Angeles\\
  \Name{Wotao Yin} \Email{	wotaoyin@math.ucla.edu}\\
  \addr   University of California, Los Angeles
 \thanks{
 The work of T. Chen was partially supported by National
Science Foundation under the project NSF 2047177 and the RPI-IBM Artificial Intelligence Research Collaboration (AIRC). The work of Y. Sun was partially supported by ONR Grant N000141712162 and AFOSR MURI FA9550-18-1-0502.}
	}

\begin{document}

\maketitle


\begin{abstract}
Stochastic compositional optimization generalizes classic (non-compositional) stochastic optimization to the minimization of compositions of functions. Each composition may introduce an additional expectation. The series of expectations may be nested. Stochastic compositional optimization is gaining popularity in applications such as reinforcement learning and meta learning. 
This paper presents a new \textbf{S}tochastically \textbf{C}orrected \textbf{S}tochastic \textbf{C}ompositional gradient method (\textbf{SCSC}).
SCSC runs in a single-time scale with a single loop, uses a fixed batch size, and guarantees to converge at the same rate as the stochastic gradient descent (SGD) method for non-compositional stochastic optimization. This is achieved by making a careful improvement to a popular stochastic compositional gradient method. 
It is easy to apply SGD-improvement techniques to accelerate SCSC. This helps SCSC achieve state-of-the-art performance for stochastic compositional optimization. 
In particular, we apply Adam to SCSC, and the exhibited rate of convergence matches that of the original Adam on non-compositional stochastic optimization. 
We test SCSC using the model-agnostic meta-learning and risk-averse portfolio management tasks.
\end{abstract}

\section{Introduction}\label{sec.intro}

In this paper, we consider stochastic compositional optimization problems of the form
\begin{align}\label{opt0}
	&\min_{\bbtheta\in \mathbb{R}^d}~~~F(\bbtheta):=f_N\left( f_{N-1}( \cdots  f_1(\bbtheta) \cdots)\right)\\
	&\qquad{\rm with}~~~f_n(\cdot):=\mathbb{E}_{\xi_n}\left[f_n(\cdot;\xi_n)\right],\quad n=1,2,\dots,N,\nonumber
\end{align}
where {\small$\bbtheta\in \mathbb{R}^d$} is the optimization variable, {\small$f_n:\mathbb{R}^{d_n}\rightarrow\mathbb{R}^{d_{n+1}}, n=1,2,\ldots, N$} (with $d_{N+1}=1$ and $d_1=d$) are smooth but possibly nonconvex functions, and $\xi_1,\ldots, \xi_N$ are independent random variables. 
The compositional formulation \eqref{opt0} covers a broader range of applications than the classical non-compositional stochastic optimization and the empirical risk minimization problem in machine learning, e.g., \cite{bottou2016}. 
In the classical non-compositional cases, the problem is to solve {\small$\min_{\bbtheta\in \mathbb{R}^d}\,\mathbb{E}_{\xi}\left[f(\bbtheta;\xi)\right]$}, which can be formulated under \eqref{opt0} when $f_1(\bbtheta)$ is a scalar function and $f_2, \cdots, f_N$ are the scalar identity maps, e.g., {\small$d_{N+1}=d_N=\cdots=d_2=1$} and {\small$d_1=d$}. 

Problem \eqref{opt0} naturally arises in a number of other areas. In reinforcement learning, finding the value function of a given policy (often referred to as \emph{policy evaluation}) can be casted as a compositional optimization problem; see e.g., \cite{dann2014policy,wang2017mp}. In financial engineering, the risk-averse portfolio optimization can be also formulated in similar form \cite{shapiro2009}. 
A recent application of \eqref{opt0} is the \emph{model-agnostic meta learning} (MAML), which is under a broader concept of few-shot meta learning; see e.g., \cite{finn2017icml}. 
It is a powerful tool for learning a new task by using the prior experience from related tasks. 
Consider a set of empirically observed tasks collected in {\small${\cal M}:=\{1,\ldots,M\}$} drawn from a certain task distribution. By a slight abuse of notation, each task $m$ has its local data $\xi_m$ from a certain distribution, which defines its loss function as {\small$F_m(\bbtheta):=\mathbb{E}_{\xi_m}\left[f(\bbtheta;\xi_m)\right],\, m\in{\cal M}$},  
where {\small$\bbtheta\in \mathbb{R}^d$} is the parameter of a prediction model (e.g., weights in a neural network), and {\small$f(\bbtheta;\xi_n)$} is the individual loss with respect to each datum. 
 In MAML, the goal is to find a common initialization that can adapt to a desired model for a set of new tasks after taking several gradient descent steps.
Specifically, we find such initialization by solving the following \emph{one-step} MAML problem
\begin{align}\label{opt1}
	\min_{\bbtheta\in \mathbb{R}^d}~F(\bbtheta):=\frac{1}{M}\sum_{m=1}^MF_m\left(\bbtheta-\alpha \nabla F_m(\bbtheta)\right)\\
	\qquad{\rm with}~~~F_m(\bbtheta):=\mathbb{E}_{\xi_m}\left[f(\bbtheta;\xi_m)\right]\nonumber
\end{align}
where $\alpha$ is the stepsize, and {\small$\nabla F_m$} is the gradient of the loss function at task $m$. The problem \eqref{opt1} is called the one-step adaptation since the loss of each task is evaluated at the model {\small$\bbtheta-\alpha \nabla F_m(\bbtheta)$} that is updated by taking one gradient descent of the each task's loss function. It is not hard to verify that \eqref{opt1} can be formulated as the special case of \eqref{opt0} with {\small$N=2$}.

The compositional structure in \eqref{opt0} has not been fully explored to develop efficient algorithms. 
In comparison,
averaging, acceleration, and variance reduction are maturing for non-compositional stochastic optimization. \emph{Can we develop a simple yet efficient counterpart of SGD for stochastic compositional optimization?} By \emph{simplicity}, we mean the new algorithm has an update \emph{without} the techniques such as double loop, accuracy-dependent stepsizes, and increasing batch sizes that are engineered to mitigate the challenges due to the compositional structure. 
By \emph{efficiency}, we mean the new algorithm can achieve the same convergence rate or the sample complexity as SGD for stochastic non-compositional problems. This paper provides an affirmative answer for this question.

\subsection{Prior art}

To put our work in context, we review prior contributions that we group in the following categories.

\begin{table*}[t]
\footnotesize
	\begin{center}
		\begin{tabular}{c||c|c|c|c|c|c ||c}
			\hline
{\bf	Accuracy metric	}	&\multicolumn{6}{c ||}{\bf Single-loop SGD methods}      & \bf \makecell{Double-loop variance \\ reduction methods}\\
			\hline
\!\!$\frac{1}{K}\sum_{k=0}^{K-1}\EE[\|\nabla F(\bbtheta^k)\|^2]$\!\!	& SCSC     &   Adam SCSC  & \cite{wang2017mp} & \cite{wang2017jmlr} & \cite{tutunov2020comp-adam}  & \cite{ghadimi2020jopt}& \cite{hu2019efficient,zhang2019icml,zhang2019nips} \\ \hline
			\textbf{sample comlpx} & $\epsilon^{-2}$ & $\epsilon^{-2}$ & $\epsilon^{-4}$ & $\epsilon^{-2.25}$ & $\epsilon^{-2.25}$ & $\epsilon^{-2}$ &  $\epsilon^{-1.5}$   \\ \hline
			\textbf{increasing batch size} & No   & No   & No  & No  & Yes  & No   & Yes   \\ \hline
			\textbf{single loop} & Yes  & Yes   & Yes   & Yes & Yes    & Yes  & No  \\ \hline
		\end{tabular}
	\end{center}
   \caption{\small Sample complexity of related algorithms that achieve the $\epsilon$-stationary point of \eqref{opt0} with $N=2$.}			\label{table:comp1}
\end{table*}

\noindent\textbf{Stochastic compositional optimization.}
Non-asymptotic analysis of stochastic compositional optimization is pioneered by \cite{wang2017mp}, where a new approach called SCGD uses two sequences of stepsizes in different time scales: a slower one for updating variable $\bbtheta$, and a faster one for tracking the value of inner function(s). 
An accelerated variant of SCGD with improved convergence rate has been developed in \cite{wang2017jmlr}. 
In concurrent with our work, an adaptive and accelerated SCGD has been studied in \cite{tutunov2020comp-adam}, but the updates of \cite{wang2017jmlr,tutunov2020comp-adam} are different from ours, and thus their convergence rates are slower than ours and that of SGD for the non-compositional case.
\blue{While most of existing algorithms for stochastic compositional problems use two-timescale stepsizes, a single timescale approach was developed for the two-level compositional problems in \cite{ghadimi2020jopt}, which has been recently extended to multi-level compositions in \cite{ruszczynski2020stochastic}.} 
No convergence rate regarding the gradient norm is given in \cite{ruszczynski2020stochastic}. 

Starting from \cite{lian2017aistats}, researchers have given much attention the stochastic compositional problem \eqref{opt0} with the \emph{finite-sum structure}. 
Building upon variance-reduction techniques for non-compositional problems \cite{johnson2013,defazio2014,nguyen2017sarah,fang2018spider,zhou2018stochastic},  
variance-reduced SCGD methods have been developed in this setting under the convex  \cite{lian2017aistats,blanchet2017unbiased,devraj2019scgd,lin2018improved}, and nonconvex assumptions \cite{hu2019efficient}. 
Recent advances also include stochastic compositional optimization with a nonsmooth regularizer \cite{huo2018accelerated,zhang2019icml,zhang2019nips}. Other variants using ADMM and accelerated variance reduction methods for finite-sum compositional problems have been studied in \cite{yu2017ijcai,xu2019katyusha}. These variance reduction-based methods have impressive performance in the finite-sum compositional problems. 
While they can be applied to the stochastic compositional problems \eqref{opt0}, 
they require an \emph{increasing batch size} and run in a double-loop manner, which is not preferable in practice. 
See a comparison in sample complexity in Table \ref{table:comp1}. 

%
\noindent\textbf{Optimization for model-agnostic meta learning.}
On the other end of the spectrum, 
MAML is a popular framework that learns a good initialization from past experiences for fast adaptation to new tasks \cite{finn2017icml,finn2019online}. MAML has been applied to various domains including reinforcement learning \cite{liu2019icml}, recommender systems, and communication \cite{simeone2020maml}. 
Due to the specific formulation, solving MAML requires information on the stochastic Hessian matrix, which can be costly in practice. 
Some recent efforts have been made towards developing Hessian-free methods for MAML; see also e.g., \cite{nichol2018first,fallah2019maml,khodak2019icml,song2020iclr,fallah2020maml,ravi2017iclr}. While most of existing works aim to find the initialization for the one-step gradient adaptation, the general multi-step MAML has also been recently studied in \cite{ji2020feb} with improved empirical performance. 
However, these methods do not fully embrace the compositional structure of MAML, and thus either lead to suboptimal sample complexity or only obtain inexact convergence for \eqref{opt1}. While this paper does not deal with Hessian-free update, our algorithms can friendly incorporate these advanced techniques motivated by application-specific challenges as well.

\subsection{Our contributions}

In this context, the present paper puts forward a new stochastic compositional gradient framework that introduces a stochastic correction to the original stochastic compositional gradient method \cite{wang2017mp}, which justifies its name \textbf{S}tochastically \textbf{C}orrected \textbf{S}tochastic \textbf{C}ompositional gradient (\textbf{SCSC}). Compared to the existing stochastic optimization schemes, our contributions can be summarized as follows. 

\textbf{c1)} We develop a stochastic gradient method termed SCSC for stochastic compositional optimization by using stochastically corrected compositional gradients. SCSC is simple to use as its alternatives, yet it achieves the same order of convergence rate {\small${\cal O}(k^{-\frac{1}{2}})$} as SGD for non-compositional problems;

\textbf{c2)} We generalize our SCSC algorithm to solve the multi-level stochastic compositional problems, and develop its adaptive gradient schemes based on the Adam-type update, both of which achieve the same order of convergence rate as their counterparts for non-compositional problems; and,

\textbf{c3)} We empirically verify the effectiveness of our SCSC-based algorithms in the portfolio management and MAML tasks using standard datasets. Comparing with the existing algorithms, our new algorithms converge faster and require a fixed (rather than increasing) batch size. 


\section{A New Method for Stochastic Compositional Optimization}
\label{sec.scg}
\subsection{Warm up: Two-level compositional problems}

For the notational brevity, we first consider a special case of \eqref{opt0} - the two-level compositional problem
\begin{align}\label{opt0-2}
	\min_{\bbtheta\in \mathbb{R}^d}~f(g(\bbtheta))=\mathbb{E}_{\xi}\left[f\left(\mathbb{E}_{\phi}[g(\bbtheta;\phi)];\xi\right)\right]
\end{align}
where $\xi$ and $\phi$ are two random variables. 
Connecting the notations of \eqref{opt0-2} with those in \eqref{opt0}, they are {\small$f_2(\,\cdot\,;\xi_2):=f(\,\cdot\,;\xi)$} and {\small$f_1(\bbtheta;\xi_1):=g(\bbtheta;\phi)$}.


Before introducing our approach, we first highlight the inherent challenge of applying the standard SGD method to \eqref{opt0}. 
When the distributions of $\phi$ and $\xi$ are unknown, the stochastic approximation \cite{robbins1951} leads to the following stochastic update 
\begin{equation}\label{eq.challe}
	\bbtheta^{k+1}=\bbtheta^k-\alpha \nabla g(\bbtheta^k;\phi^k) \nabla f(\mathbb{E}_{\phi}[g(\bbtheta^k;\phi)];\xi^k) 
\end{equation}	
where $\phi^k$ and $\xi^k$ are samples drawn at iteration $k$. Notice that obtaining the unbiased stochastic gradient {\small$\nabla g(\bbtheta^k;\phi^k) \nabla f(\mathbb{E}_{\phi}[g(\bbtheta^k;\phi)];\xi^k)$} is still costly since the gradient $\nabla f$ is evaluated at {\small$\mathbb{E}_{\phi}[g(\bbtheta^k;\phi)]$}. Except that the gradient $\nabla f$ is linear, the expectation in \eqref{eq.challe} cannot be omitted, because the stochastic gradient {\small$\nabla g(\bbtheta^k;\phi^k) \nabla f(g(\bbtheta^k;\phi^k);\xi^k)$} is biased, i.e.,
\begin{align}\label{eq.challe-2}
\mathbb{E}_{\phi^k, \xi^k}[\nabla g(\bbtheta^k;\phi^k) \nabla f(g(\bbtheta^k;\phi^k);\xi^k)] \neq	\mathbb{E}_{\phi,\xi}\left[\nabla g(\bbtheta^k;\phi) \nabla f(\mathbb{E}_{\phi}[g(\bbtheta^k;\phi)];\xi)\right].
\end{align} 
Therefore, the machinery of stochastic gradient descent cannot be directly applied here. 

\begin{algorithm}[t]
\caption{SCSC for two-level problem}\label{alg:scg}
    \begin{algorithmic}[1]
    \State{\textbf{initialize:}~$\bbtheta^0$, $\bby^0$, stepsizes $\alpha_0, \beta_0$}
    \For{$k= 1, 2,\ldots, K$}
           \State{randomly select datum $\phi^k$} 
           \State{compute $g(\bbtheta^k;\phi^k)$ and $\nabla g(\bbtheta^k;\phi^k)$}
            \State{update variable $\bby^{k+1}$ via \eqref{eq.SCSC-2} or \eqref{eq.SCSC-3}}
            \State{randomly select datum $\xi^k$} 
            \State{compute $\nabla f(\bby^{k+1};\xi^k)$} 
            \State{update variable $\bbtheta^{k+1}$ via \eqref{eq.SCSC-1}}
    \EndFor
    \end{algorithmic}
\end{algorithm}

To overcome this difficulty, a popular SCGD has been developed in \cite{wang2017mp} for solving the two-level stochastic compositional problem \eqref{opt0-2}, which is given by 
\begin{subequations}\label{eq.wSCGD}
\begin{align}
	\bby^{k+1}&=(1-\beta_k)\bby^k+\beta_k g(\bbtheta^k;\phi^k)\label{eq.wSCGD-1}\\
	\bbtheta^{k+1}&=\bbtheta^k-\alpha_k \nabla g(\bbtheta^k;\phi^k) \nabla f(\bby^{k+1};\xi^k)\label{eq.wSCGD-2}
\end{align}
\end{subequations}
where $\alpha_k$ and $\beta_k$ are two sequences of decreasing stepsizes. The above recursion involves two iterates, $\bby^k$ and $\bbtheta^k$, whose updates are coupled with each other. To ensure convergence, SCGD requires $\bby^k$ to be updated in a timescale asymptotically faster than that of $\bbtheta^k$ so that $\bbtheta^k$ is relatively static with respect to $\bby^k$; i.e., {\small$\lim\limits_{k\rightarrow \infty}\alpha_k/\beta_k=0$}. This prevents SCGD from choosing the same stepsize as SGD for the non-compositional stochastic problems, and also results in its \emph{suboptimal convergence rate}. 
In \eqref{eq.wSCGD-1}, the iterate $\bby^{k+1}$ linearly combines $\bby^k$ and {\small$g(\bbtheta^k;\phi^k)$}, where $\bby^k$ is updated by the outdated iterate {\small$\bbtheta^{k-1}$}. We notice that this is the main reason of using a smaller stepsize $\alpha_k$ in the proof of \cite{wang2017mp}. 

With more insights given in Section \ref{subsec.ode}, our new method that we term stochastically corrected stochastic compositional gradient (\textbf{SCSC}) addresses this issue by linearly combining a ``corrected'' version of $\bby^k$ and $g(\bbtheta^k;\phi^k)$. 
\blue{Since we use {\small$\nabla f(\bby^{k+1};\xi^k)$} to approximate {\small$\nabla f(g(\bbtheta^k);\xi^k)$}, we want {\small$\bby^{k+1}$} as close to {\small$g(\bbtheta^k)$} as possible.
Roughly speaking, if {\small$\bby^k\approx g(\bbtheta^{k-1})$}, we gauge that {\small$g(\bbtheta^k)\approx g(\bbtheta^{k-1})+\nabla g(\bbtheta^k;\phi^k)(\bbtheta^k-\bbtheta^{k-1})$}.} 
Therefore, we propose the following new update 
\begin{subequations}\label{eq.SCSC}
\begin{align}
	\bbtheta^{k+1}&=\bbtheta^k-\alpha_k \nabla g(\bbtheta^k;\phi^k) \nabla f(\bby^{k+1};\xi^k) \label{eq.SCSC-1}\\
		\bby^{k+1}&=(1-\beta_k)\left(\bby^k+\nabla g(\bbtheta^k;\phi^k)(\bbtheta^k-\bbtheta^{k-1})\right)+\beta_k g(\bbtheta^k;\phi^k).\label{eq.SCSC-2}
\end{align}

We can also approximate {\small$\nabla g(\bbtheta^k;\phi^k)(\bbtheta^k-\bbtheta^{k-1})$} by the first-order Taylor expansion, that is
\begin{align}\label{eq.SCSC-3}
	\bby^{k+1}&=(1-\beta_k)\left(\bby^k+g(\bbtheta^k;\phi^k)-g(\bbtheta^{k-1};\phi^k)\right)+\beta_k g(\bbtheta^k;\phi^k).
\end{align}
\end{subequations} 

Different from \eqref{eq.wSCGD}, we use two sequences of stepsizes $\alpha_k$ and $\beta_k$ in \eqref{eq.SCSC} that decrease at the same rate as SGD. 
As we will show later, under a slightly different assumption, both \eqref{eq.SCSC-2} and \eqref{eq.SCSC-3} can guarantee that the new approach achieves the same convergence rate ${\cal O}(k^{-\frac{1}{2}})$ as SGD for the non-compositional stochastic optimization problems. 

  \vspace{0.1cm}
\noindent\textbf{Choices of \eqref{eq.SCSC-2} and \eqref{eq.SCSC-3}.}
The two choices of \eqref{eq.SCSC-2} and \eqref{eq.SCSC-3} for updating $\bby^k$ have different advantages.
At each iteration, the main cost of \eqref{eq.SCSC-2} is  one function evaluation, $g(\bbtheta^k;\phi^k)$. Although it needs $\nabla g(\bbtheta^k;\phi^k)$, it is already computed in the update of $\bbtheta^k$. 
In comparison, the main cost of \eqref{eq.SCSC-3} is two function evaluations, $g(\bbtheta^k;\phi^k)$ and $g(\bbtheta^{k-1};\phi^k)$. Therefore,  \eqref{eq.SCSC-3} has a higher cost. However, in some applications such as MAML with neural network parameterization, \eqref{eq.SCSC-3} is a better choice since one often avoids computing $\nabla g(\bbtheta^k;\phi^k)$, which is the Hessian matrix in MAML, but instead computes the matrix-vector (Hessian-gradient) product $\nabla g(\bbtheta^k;\phi^k) \nabla f(y^{k+1};\xi^k)$ in \eqref{eq.SCSC-1} and can evaluate $g(\bbtheta^k;\phi^k)$ and $g(\bbtheta^{k-1};\phi^k)$ at a relatively low cost.  

\subsection{Algorithm development motivated by ODE analysis.}\label{subsec.ode} 

We provide some intuition of our design via an ODE-based construction for the corresponding deterministic continuous-time system. To achieve so, we make the following assumptions \cite{wang2017mp,lian2017aistats,zhang2019nips}. 

\noindent\textbf{Assumption 1.}
\emph{Functions $f$ and $g$ are $L_f$- and $L_g$-smooth, that is, for any $\bbtheta, \bbtheta'\in\mathbb{R}^d$, we have} {\small$\|\nabla f(\bbtheta;\xi)-\nabla f(\bbtheta';\xi)\|\leq L_f\|\bbtheta-\bbtheta'\|,~\|\nabla g(\bbtheta;\phi)-\nabla g(\bbtheta';\phi)\|\leq L_g\|\bbtheta-\bbtheta'\|$}.

\noindent\textbf{Assumption 2.}
\emph{The stochastic gradients of $f$ and $g$ are bounded in expectation, that is {\small$\mathbb{E}\left[\|\nabla g(\bbtheta;\phi)\|^2\right]\leq C_g^2$ and $\mathbb{E}\left[\|\nabla f(\bby;\xi)\|^2\right]\leq C_f^2$}.}

Assumptions 1 and 2 require both the function values and the gradients to be Lipschitz continuous. 
As a result, the compositional function {\small$F(\bbtheta)$} is also smooth with \cite{zhang2019nips} 
\begin{equation}\label{eq.smooth_const}
    L:=C_g^2 L_f+C_f L_g.
\end{equation}

Let $t$ be time in this subsection. Consider the following ODE
\begin{equation}\label{eq.ode-theta}
    \dot{\bbtheta}(t)=-\alpha \nabla g(\bbtheta(t))\nabla f(\bby(t))
\end{equation}
where the constant $\alpha>0$. If we set {\small$\bby(t)=g(\bbtheta(t))$}, then this system describes a gradient flow that monotonically decreases {\small$f\left(g(\bbtheta(t))\right)$}. In this case, we have {\small$\frac{d}{dt}f\left(g(\bbtheta(t))\right)=\langle \nabla g(\bbtheta(t))\nabla f(g(\bbtheta(t))), \dot{\bbtheta}(t)\rangle=-\frac{1}{\alpha}\|\dot{\bbtheta}(t)\|^2$}. 
However, if we can evaluate gradient $\nabla f$ only at {\small$\bby(t)\neq g(\bbtheta(t))$}, it introduces inexactness and thus {\small$f\left(g(\bbtheta(t))\right)$} may lose monotonicity, namely
\begin{align}\label{eq.ode-fg}
 	\frac{d}{dt}f\left(g(\bbtheta(t))\right) 
	&\stackrel{(a)}{=}-\frac{1}{\alpha}\|\dot{\bbtheta}(t)\|^2+\langle\nabla g(\bbtheta(t))\big(\nabla f(g(\bbtheta(t)))-\nabla f(\bby(t))\big), \dot{\bbtheta}(t)\rangle\nonumber\\
	&\stackrel{(b)}{\leq}-\frac{1}{\alpha}\|\dot{\bbtheta}(t)\|^2+\| \nabla g(\bbtheta(t))\|\|\nabla f(g(\bbtheta(t)))-\nabla f(\bby(t))\| \|\dot{\bbtheta}(t)\|\nonumber\\
	&\stackrel{(c)}{\leq}-\frac{1}{2\alpha}\|\dot{\bbtheta}(t)\|^2+\frac{\alpha C_g^2L_f^2}{2}\|g(\bbtheta(t))-\bby(t)\|^2
	\end{align} 
\hspace{-3pt}where (a) follows from \eqref{eq.ode-theta}, (b) uses the Cauchy-Schwarz inequality, (c) is due to Assumptions 1 and 2 as well as the Young's inequality. In general, the RHS of \eqref{eq.ode-fg} is not necessarily negative. 
Therefore, it motivates an energy function with both {\small$f(g(\bbtheta(t)))$} and {\small$\|g(\bbtheta(t))-\bby(t)\|^2$}, given by
\begin{equation}\label{eq.ct-lyap}
	{\cal V}(t):=f(g(\bbtheta(t)))+\|g(\bbtheta(t))-\bby(t)\|^2.
\end{equation}
We wish ${\cal V}(t)$ would monotonically decrease. By substituting the bound in \eqref{eq.ode-fg}, we have
\begin{align}\label{eq.ode-lyap}
	\dot{\cal V}(t)
	&\leq-\frac{1}{2\alpha}\|\dot{\bbtheta}(t)\|^2+\frac{\alpha C_g^2L_f^2}{2}\|g(\bbtheta(t))-\bby(t)\|^2
 +2\left\langle \bby(t)-g(\bbtheta(t)), \dot{\bby}(t)-\nabla g(\bbtheta(t)) \dot{\bbtheta}(t)\right\rangle\nonumber\\
	&=-\frac{1}{2\alpha}\|\dot{\bbtheta}(t)\|^2-\Big(2\beta-\frac{\alpha C_g^2L_f^2}{2}\Big)\|g(\bbtheta(t))-\bby(t)\|^2\nonumber\\
&\quad~+ 2\left\langle \bby(t)-g(\bbtheta(t)), \dot{\bby}(t)\!+\!\beta(\bby(t)\!-\!g(\bbtheta(t)))\!-\!\nabla g(\bbtheta(t)) \dot{\bbtheta}(t)\right\rangle
	\end{align}
\hspace{-3pt}where $\beta>0$ is a fixed constant. The first two terms in the RHS of \eqref{eq.ode-lyap} are non-positive given that {\small$\alpha\geq 0$} and {\small$\beta\geq {\alpha C_g^2L_f^2}/{4}$}, but the last term can be either positive or negative. Following the \textbf{maximum descent principle} of ${\cal V}(t)$, we are motivated to use the following dynamics
\begin{equation}\label{eq.ode-y}
	\dot{\bby}(t)=-\beta\left(\bby(t)-g(\bbtheta(t))\right)+\nabla g(\bbtheta(t)) \dot{\bbtheta}(t)~\Longrightarrow~\dot{\cal V}(t)\leq 0.
\end{equation}
Directly implementing \eqref{eq.ode-y} in the discrete time is intractable. Instead, we approximate the continuous-time update by either the backward difference or the Taylor expansion, given by
\begin{align}\label{eq.appode-y}
\nabla g(\bbtheta(t)) \dot{\bbtheta}(t)\approx \gamma_k \nabla g(\bbtheta^k) \Big(\bbtheta^k-\bbtheta^{k-1}\Big) ~~~~~{\rm or}~~~\approx\gamma_k \Big(g(\bbtheta^k)-g(\bbtheta^{k-1})\Big)
\end{align}
where $k$ is the discrete iteration index, and $\gamma_k>0$ is the weight controlling the approximation. 

With the insights gained from \eqref{eq.ode-theta} and \eqref{eq.ode-y}, 
our stochastic update \eqref{eq.SCSC} essentially discretizes time $t$ into iteration $k$, and replaces the exact function $g(\bbtheta(t))$ and the gradients {\small$\nabla g(\bbtheta(t)), \nabla f(\bby(t))$} by their stochastic values. The choice {\small$\gamma_k:=1-\beta_k$} in \eqref{eq.appode-y} will simplify some constants in the proof.

\noindent\textbf{Connection to existing methods.} 
Using this interpretation, the dynamics of $\bby(t)$ in SCGD \cite{wang2017mp} is 
	\begin{equation}\label{eq.ode-SCGD}
	\dot{\bby}(t)=-\beta\left(\bby(t)-g(\bbtheta(t))\right)
	\end{equation}
which will leave an additional non-negative term {\small$\langle \bby(t)-g(\bbtheta(t)), -\nabla g(\bbtheta(t)) \dot{\bbtheta}(t)\rangle\leq C_g\|\bby(t)-g(\bbtheta(t))\|\|\dot{\bbtheta}(t)\|$} in \eqref{eq.ode-lyap}. To ensure the convergence of ${\cal V}(t)$, a much smaller stepsize $\alpha$ is needed.  

Using the ODE interpretation, the dynamics of $\bby(t)$ in the recent \emph{variance-reduced} compositional gradient approaches, e.g.,  \cite{lian2017aistats,hu2019efficient,zhang2019icml,zhang2019nips} can be written as
\begin{equation}\label{eq.ode-vrSCGD}
	\dot{\bby}(t)=\nabla g(\bbtheta(t)) \dot{\bbtheta}(t)
\end{equation}
which leaves the non-negative term {\small$\|g(\bbtheta(t))-\bby(t)\|^2$} uncancelled in \eqref{eq.ode-lyap}. Therefore, to ensure convergence of ${\cal V}(t)$, the variance-reduced compositional approaches must calculate the \emph{full gradient} {\small$\nabla f(g(\bbtheta(t)))$} periodically to erase the error accumulated by {\small$\|g(\bbtheta(t))-\bby(t)\|^2$}.

 \vspace{0.1cm}
\noindent\textbf{Comparison with \cite{ghadimi2020jopt}.}
The recent work \cite{ghadimi2020jopt} introduces the first algorithm NASA that achieves the same rate of SGD; so does SCSC. Both NASA in \cite{ghadimi2020jopt} and SCSC of this paper are single-time scale algorithms. 
There are, however, differences and advantages of SCSC over NASA: 

{\bf D1)} SCSC is \emph{simpler} and appears to be \emph{easier to generalize}. The $\bbtheta$-update of SCSC closely resembles SGD; that of NASA is more complicated. Specifically, NASA adds an extra sequence to the $\bbtheta$-update \eqref{eq.wSCGD-2} to reduce the variance of $\mathbb{E}[\nabla g(\bbtheta^k;\phi^k) \nabla f(\bby^{k+1};\xi^k)]$. SCSC achieves the same rate just generating a better sequence $\bby^{k+1}$, maintaining the SGD-like update. This allows us to apply SGD techniques such as Adam in a \emph{plug-and-play} manner to SCSC and inherit the benefits. It is unclear whether we can do the same to NASA.
We believe that one can apply variance reduction techniques to the $\bbtheta$-update to further improve SCSC.

{\bf D2)} SCSC is accompanied by a new \emph{ODE analysis} that explains in a couple of equations our design intuition and
the roles of SCSC's key iterates. The continuous-time analog appeals to the applied math and control communities and may encourage them to make further generalizations. 

{\bf D3)} We numerically compared SCSC with NASA. The results indicate that SCSC is empirically more stable and also more robust to the choice of stepsizes.

\section{Adam-type and Multi-level Variants}
In this section, we introduce two variants of our new stochastic compositional gradient method:  adaptive stochastic gradient and multi-level compositional gradient schemes.

\subsection{Adam-type adaptive gradient approach}

When the sought parameter $\bbtheta$ represents the weight of a neural network, in the non-compositional stochastic problems, finding a good parameter $\bbtheta$ will be much more efficient if adaptive SGD approaches are used such as AdaGrad \cite{duchi2011adaptive} and Adam \cite{kingma2014adam}. 
We first show that our SCSC method can readily incorporate Adam update for $\bbtheta$, and establish that it achieves the same convergence rate as the original Adam approach for the non-compositional stochastic problems \cite{reddi2019adam,chen2019adam}.

Following the Adam and AMSGrad in \cite{kingma2014adam,reddi2019adam,chen2019adam}, the Adam SCSC approach uses two sequences $\bbh^k$ and $\bbv^k$ to track the exponentially weighted gradient of $\bbtheta^k$ and its second moment estimates, and uses $\bbv^k$ to inversely weight the gradient estimate $\bbh^k$. 
The update can be written as  
\begin{subequations}\label{eq.adaSCGD}
	\begin{align}
	\bbh^{k+1}&=\eta_1 \bbh^k+(1-\eta_1)\bm\nabla^k \label{eq.adaSCGD-1}\\
	\bbv^{k+1}&=\eta_2 \hat \bbv^k+(1-\eta_2) (\bm\nabla^k)^2\label{eq.adaSCGD-2}\\
		\bbtheta^{k+1}&=\bbtheta^k-\alpha_k \frac{\bbh^{k+1}}{\sqrt{\bm{\epsilon}+\hat \bbv^{k+1}}} \label{eq.adaSCGD-3}\\
 \bby^{k+1}&~{\rm via}~\eqref{eq.SCSC-2}~{\rm or}~\eqref{eq.SCSC-3}\notag
\end{align}
\end{subequations} 
where the gradient is defined as {\small$\bm\nabla^k:=\nabla g(\bbtheta^k;\phi^k) \nabla f(\bby^{k+1};\xi^k)$; $\hat \bbv^{k+1}:=\max\{\bbv^{k+1}, \hat \bbv^k\}$} ensures the monotonicity of the scaling factor in \eqref{eq.adaSCGD-3}; the constant vector is $\bbepsilon>0$; and $\eta_1$ and $\eta_2$ are two exponential weighting parameters. The vector division and square in \eqref{eq.adaSCGD} are defined element-wisely.

The key difference of the Adam-SCSC relative to the original Adam is that the stochastic gradient $\bm\nabla^k$ used in the updates \eqref{eq.adaSCGD-1} and \eqref{eq.adaSCGD-2} is not an unbiased estimate of the true one $\nabla F(\bbtheta^k)$. Furthermore, the gradient bias incurred by the Adam update intricately depends on the multi-level compositional gradient estimator, the analysis of which is not only challenging but also of its independent interest.

\subsection{Multi-level compositional problems}\label{sec.multi}

Aiming to solve practical problems with more general stochastic compositional structures, we extend our SCSC method in Section \ref{sec.scg} for \eqref{opt0-2} to the multi-level problem \eqref{opt0}.
As an example, the \emph{multi-step} MAML problem \cite{ji2020feb} can be formulated as the multi-level compositional problem \eqref{opt0}. 
In this case, a globally shared initial model $\bbtheta$ for the $N$-step adaptation can be found by solving 
\begin{align}\label{opt2}
\small
	&\min_{\bbtheta\in \mathbb{R}^d}~F(\bbtheta):=\frac{1}{M}\sum_{m=1}^MF_m\left(\tilde{\bbtheta}^N_m(\bbtheta)\right)\\
	&~~~{\rm with}~~~\tilde{\bbtheta}^{n+1}_m=\tilde{\bbtheta}^n_m-\alpha \nabla F_m(\tilde{\bbtheta}^n_m)~~
	~{\rm recursively}\nonumber
\end{align}
where $\tilde{\bbtheta}^N_m(\bbtheta)$ is obtained after taking $N$ step gradient descent on task $m$ and initialized with $\tilde{\bbtheta}^0_m=\bbtheta$. 

Different from SCSC for the two-level compositional problem \eqref{opt0-2}, the multi-level SCSC (\textbf{multi-SCSC}) requires to track $N-1$ functions $f_1, \cdots, f_{N-1}$ using $\bby_1, \cdots, \bby_{N-1}$. 
Following the tracking update of SCSC, the multi-SCSC update is 
\begin{subequations}\label{eq.SCSCm}
\small
	\begin{align}
		\bby_1^{k+1}&=(1-\beta_k)\bby_1^k+\beta_k f_1(\bbtheta^k;\xi_1^k) +(1-\beta_k)(f_1(\bbtheta^k;\xi_1^k)-f_1(\bbtheta^{k-1};\xi_1^k))\label{eq.SCSCm-1}\\
	    &\cdots \nonumber \\
		\bby_{N-1}^{k+1}&=(1-\beta_k)\bby_{N-1}^k+\beta_k f_{N-1}(\bby_{N-2}^{k+1};\xi_{N-1}^k) +(1-\beta_k)(f_{N-1}(\bby_{N-2}^{k+1};\xi_{N-1}^k)-f_{N-1}(\bby_{N-2}^k;\xi_{N-1}^k))\label{eq.SCSCm-2}\\		
	\bbtheta^{k+1}&=\bbtheta^k- \alpha_k  \nabla f_1(\bbtheta^k;\xi_1^k)\cdots\nabla f_{N-1}(\bby_{N-2}^{k+1};\xi_{N-1}^k)\nabla f_N(\bby_{N-1}^{k+1};\xi_N^k). \label{eq.SCSCm-3}
\end{align}
\end{subequations}
Note that both \eqref{eq.SCSC-2} and \eqref{eq.SCSC-3} can be used in multi-SCSC \eqref{eq.SCSCm}, though above we choose \eqref{eq.SCSC-3}. Multi-SCSC can also incorporate Adam-type update. 
Analyzing multi-SCSC is more challenging that SCSC, since the tracking variables are statistically dependent on each other. Specifically, conditioned on the randomness up to iteration $k$, the variable {\small$\bby_n^{k+1}$} depends on {\small$\bby_{n-1}^{k+1}$} and thus also on {\small$\bby_{n-2}^{k+1},\cdots,\bby_1^{k+1}$}.
Albeit its complex compositional form, as we will show later, multi-SCSC also achieves the same rate of convergence as SGD for non-compositional stochastic optimization. 

\begin{algorithm}[t]
\caption{Adam SCSC method}\label{alg:adascg}
    \begin{algorithmic}[1]
    \State{\textbf{initialize:}~$\bbtheta^0$, $\bby^0$, $\bbv^0$, $\bbh^0$, $\eta_1$, $\eta_2$, $\alpha_0, \beta_0$}
    \For{$k= 1, 2,\ldots, K$}
           \State{randomly select datum $\phi^k$} 
           \State{compute $g(\bbtheta^k;\phi^k)$ and $\nabla g(\bbtheta^k;\phi^k)$}
            \State{update variable $\bby^{k+1}$ via \eqref{eq.SCSC-2} or \eqref{eq.SCSC-3}}
            \State{randomly select datum $\xi^k$} 
            \State{compute $\nabla f(\bby^{k+1};\xi^k)$} 
            \State{update $\bbh^{k+1}, \bbv^{k+1}, \bbtheta^{k+1}$ via \eqref{eq.adaSCGD}}
    \EndFor
    \end{algorithmic}
\end{algorithm}




\section{Convergence Analysis of SCSC}
\label{sec.ic-ana}
In this section, we establish the convergence of all SCSC algorithms. 
For our analysis, in addition to Assumptions 1 and 2, we make the following assumptions.

\noindent\textbf{Assumption 3.} \emph{Sampling oracle satisfies that i) {\small$\mathbb{E}\left[g(\bbtheta;\phi^k)\right]=g(\bbtheta)$, and, ii) $\mathbb{E}\left[\nabla g(\bbtheta;\phi^k) \nabla f(\bby;\xi^k)\right]=\nabla g(\bbtheta) \nabla f(\bby)$}.}

\noindent\textbf{Assumption 4.} \emph{Function $g(\bbtheta;\phi^k)$ has bounded variance, i.e.,} {\small$\EE\left[\|g(\bbtheta;\phi^k)-g(\bbtheta)\|^2\right]\leq V_g^2$}.

Assumptions 3 and 4 are standard in stochastic compositional optimization; e.g., \cite{wang2017mp,wang2017jmlr,lian2017aistats,zhang2019nips}, and are analogous to the unbiasedness and variance assumptions for stochastic non-compositional problems. Note that the independence of $\phi^k$ and $\xi^k$ is sufficient but not necessary for Assumption 3.

\subsection{Convergence in the two-level case}\label{subsec.two}

With insights gained from the continuous-time Lyapunov function \eqref{eq.ct-lyap}, 
our analysis in this subsection critically builds on the following discrete-time Lyapunov function:
 \begin{equation}\label{eq.Lyap}
 \small
 	{\cal V}^k:=F(\bbtheta^k)-F(\bbtheta^*)+ \|g(\bbtheta^{k-1})-\bby^k\|^2
 \end{equation}
 where {\small$\bbtheta^*$} is the optimal solution of the problem \eqref{opt0-2}.
 
 \begin{lemma}[Tracking variance of SCSC]
\label{lemma2}
Consider {\small${\cal F}^k$} as {\small${\cal F}^k:=\left\{\phi^0, \ldots, \phi^{k-1}, \xi^0, \ldots, \xi^{k-1}\right\}$}. 
Suppose Assumptions 1-4 hold, and $\bby^{k+1}$ is generated by running SCSC iteration \eqref{eq.SCSC-1} and \eqref{eq.SCSC-3} conditioned ${\cal F}^k$. The mean square error of $\bby^{k+1}$ satisfies 
{\small\begin{align}\label{eq.lemma2}
 \EE\left[\|g(\bbtheta^k)-\bby^{k+1}\|^2\mid{\cal F}^k\right]\leq (1-\beta_k)^2\|g(\bbtheta^{k-1})-\bby^k\|^2  +2(1-\beta_k)^2C_g^2\|\bbtheta^k-\bbtheta^{k-1}\|^2+2\beta_k^2V_g^2.
\end{align}}
\end{lemma}
Intuitively, since {\small$\|\bbtheta^k-\bbtheta^{k-1}\|^2={\cal O}(\alpha_{k-1}^2)$}, Lemma \ref{lemma2} implies that if the stepsizes {\small$\alpha_k^2$} and {\small$\beta_k^2$} are decreasing, the mean square error of $\bby^{k+1}$ will decrease. Note that Lemma \ref{lemma2} presents the performance of $\bby^{k+1}$ using the update \eqref{eq.SCSC-3}. If we use the update \eqref{eq.SCSC-2} instead, the bound in \eqref{eq.lemma2} will have an additional term {\small$(1-\beta_k)^2\|\bbtheta^k-\bbtheta^{k-1}\|^4$}. Under a stronger version of Assumption 2 (e.g., fourth moments), the remaining analysis still follows; see the derivations in supplementary material. 

Building upon Lemma \ref{lemma2}, we establish the following theorem.
 
\begin{theorem}[two-level SCSC]\label{theorem1}
Under Assumptions 1-4, if we choose the stepsizes as {\small$\alpha_k=\frac{2\beta_k}{C_g^2L_f^2}=\alpha=\frac{1}{\sqrt{K}}$}, the iterates $\{\bbtheta^k\}$ of SCSC in Algorithm \ref{alg:scg} satisfy
\begin{equation}\label{eq.theorem1}
 \small
    \frac{\sum_{k=0}^{K-1}\EE[\|\nabla F(\bbtheta^k)\|^2]}{K}\leq\frac{2{\cal V}^0+2B_1}{\sqrt{K}}
\end{equation}
where the constant is defined as {\small$B_1:=\frac{L}{2}C_g^2C_f^2+4V_g^2+16C_g^4C_f^2$}.
\end{theorem}


Theorem \ref{theorem1} implies that the convergence rate of SCSC is {\small${\cal O}(k^{-\frac{1}{2}})$}, which is on the same order of SGD's convergence rate for the stochastic non-compositional nonconvex problems \cite{ghadimi2013sgd}, and significantly improves {\small${\cal O}(k^{-\frac{1}{4}})$} of the original SCGD \cite{wang2017mp} and {\small${\cal O}(k^{-\frac{4}{9}})$} of its accelerated version \cite{wang2017jmlr}. 
Comparing with \cite{ghadimi2020jopt,ruszczynski2020stochastic} that achieves the same rate of {\small${\cal O}(k^{-\frac{1}{2}})$} for the two-level problem, our algorithm is simpler which makes it possible to adopt the Adam update. 
In addition, this rate is not directly comparable to those under  variance-reduced compositional methods, e.g., \cite{lian2017aistats,hu2019efficient,zhang2019icml,zhang2019nips} since SCSC does not need the increasing batchsize nor double-loop. 

\subsection{Convergence of Adam-SCSC}\label{subsec.adam}

The convergence analysis for Adam SCSC builds on the following Lyapunov function:
 \begin{align}\label{eq.Lyap1}
 	{\cal V}^k:=&F(\bbtheta^k)\!-\!F(\bbtheta^*)\!-\!\sum\limits_{j=k}^{\infty}\eta_1^{j-k+1}\alpha_j\left\langle \nabla F(\bbtheta^{k-1}), \frac{\bbh^k}{\sqrt{\bm{\epsilon}+\hat \bbv^k}} \right\rangle + c\left\|g(\bbtheta^{k-1})-\bby^k\right\|^2
 \end{align}
 where $c$ is a constant that depends on $\eta_1, \eta_2$ and $\epsilon$. 
Clearly, the Lyapunov function \eqref{eq.Lyap1} is a generalization of \eqref{eq.Lyap} for SCSC, which takes into account the adaptive gradient update by subtracting the inner product between the full gradient and the Adam SCSC update.
 Intuitively, if the adaptive stochastic gradient direction is aligned with the gradient direction, this term will also become small. 
 
 To establish the convergence of Adam SCSC, we need a slightly stronger version of Assumption 2, which is standard in analyzing the convergence of Adam \cite{kingma2014adam,reddi2019adam,chen2019adam}.
 
\noindent\textbf{Assumption 5.} \emph{Stochastic gradients are bounded almost surely,} {\small$\left\|\nabla g(\bbtheta;\phi)\right\|\leq C_g, \left\|\nabla f(\bby;\xi)\right\|\leq C_f$}.

Analogous to Theorem \ref{theorem1}, we establish the convergence of Adam SCSC under nonconvex settings. 

\begin{theorem}[Adam SCSC]\label{theorem4}
Under Assumptions 1 and 3-5, if we choose the parameters {\small$\eta_1<\sqrt{\eta_2}<1$}, and the stepsizes as {\small$\alpha_k=2\beta_k=\frac{1}{\sqrt{K}}$}, the iterates {\small$\{\bbtheta^k\}$} of Adam SCSC satisfy
{\small\begin{align}\label{eq.theorem4}
\frac{1}{K}\sum_{k=0}^{K-1}\EE[\|\nabla F(\bbtheta^k)\|^2]\leq& \frac{2(\epsilon+C_g^2C_f^2)^{\frac{1}{2}}}{(1-\eta_1)}\times \nonumber\\
  &\bigg(\frac{C_gC_fd\tilde{\eta}}{K} +\frac{{\cal V}^0\!+\!(4C_g^2\tilde{\eta}+V_g^2)c \!+\! 2d\tilde{\eta}L}{\sqrt{K}}\!+\!\frac{(1\!+\!(1-\eta_1)^{-1})C_g^2C_f^2d \epsilon^{-\frac{1}{2}}}{K}\!\bigg)
\end{align}}
where $d$ is the dimension of $\bbtheta$, and the constant is defined as {\small$\tilde{\eta}:=(1-\eta_1)^{-1}(1-\eta_2)^{-1}(1-\eta_1^2/\eta_2)^{-1}$}.
\end{theorem}

Theorem \ref{theorem4} implies that the convergence rate of Adam SCSC is also {\small${\cal O}(k^{-\frac{1}{2}})$}. This rate is again on the same order of Adam's convergence rate for the stochastic non-compositional nonconvex problems \cite{chen2019adam}, and significantly faster than  {\small${\cal O}(k^{-\frac{4}{9}})$} of the existing adaptive compositional SGD method \cite{tutunov2020comp-adam}. 
As a by-product, the newly designed Lyapunov function \eqref{eq.Lyap1} also significantly streamlines the original analysis of Adam under nonconvex settings \cite{chen2019adam}, which is of its independent interest. 

\subsection{Convergence of multi-SCSC} 
In this section, we establish the convergence results of the multi-level SCSC, and present the corresponding analysis. 

The subsequent analysis for the multi-level problem builds on the following \emph{Lyapunov function}:
 \begin{equation}\label{eq.Lyap2}
 \small
 	{\cal V}^k:=F(\bbtheta^k)-F(\bbtheta^*)+ \sum_{n=1}^{N-1} \left\|\bby_n^k-f_n(\bby_{n-1}^k)\right\|^2
 \end{equation}
 where $\bbtheta^*$ is the optimal solution of the problem \eqref{opt0}.

To this end, we need a generalized version of Assumptions 1-4 for the multi-level setting. 

\noindent\textbf{Assumption m1.}
\emph{Functions $\{f_n\}$ are $L_n$-smooth, that is, for any $\bbtheta, \bbtheta'\in\mathbb{R}^d$,} {\small$\|\nabla f_n(\bbtheta;\xi_n)-\nabla f_n(\bbtheta';\xi_n)\|\leq L_n\|\bbtheta-\bbtheta'\|$}.

\noindent\textbf{Assumption m2.}
\emph{The stochastic gradients $\{\nabla f_n\}$ are bounded, i.e., {\small$\mathbb{E}\left[\|\nabla f_n(\bbtheta;\xi_n)\|^2\right]\leq C_n^2$}.}

\noindent\textbf{Assumption m3.} \emph{Sampling oracle satisfies that {\small$\mathbb{E}\left[f_n(\bbtheta;\xi_n^k)\right]=f_n(\bbtheta),\,\forall n$, and $\mathbb{E}\left[\nabla f_1(\bbtheta;\xi_1^k)\cdots\nabla f_N(\bby_{N-1};\xi_N^k)\right]=\nabla f_1(\bbtheta)\cdots\nabla f_N(\bby_{N-1})$}.}

\noindent\textbf{Assumption m4.} \emph{For all $n$, $f_n(\bbtheta;\xi_n)$ has bounded variance, i.e.,} {\small$\EE\left[\|f_n(\bbtheta;\xi_n)-f_n(\bbtheta)\|^2\right]\leq V^2$}.

Building upon these assumptions, we establish the convergence of multi-SCSC. 

\begin{theorem}[multi-level SCSC]\label{theorem-mSCSC}
Under Assumptions m1-m4, if we choose the stepsizes as $\alpha_k=\frac{2\beta_k}{\sum_{n=1}^{N-1}A_n^2}=\frac{1}{\sqrt{K}}$, the iterates $\{\bbtheta^k\}$ of the multi-level SCSC iteration \eqref{eq.SCSCm} satisfy
\begin{equation}\label{eq.theorem-mSCSC}
\small
    \frac{\sum_{k=0}^{K-1}\EE[\|\nabla F(\bbtheta^k)\|^2]}{K}\leq\frac{2{\cal V}^0+2(B_2+B_3{(\sum_{n=1}^{N-1}A_n^2)^2}/{4})}{\sqrt{K}}.
\end{equation}
where $B_2, B_3, A_1, \ldots, A_N$ are some constants that depend on $C_1, \ldots, C_N$ and $L_1, \ldots, L_N$. 
\end{theorem}
Theorem \ref{theorem-mSCSC} implies that the convergence rate of multi-SCSC is also {\small${\cal O}(k^{-\frac{1}{2}})$}. This rate is again on the same order of SGD's rate for the stochastic non-compositional nonconvex problems.

\section{Numerical Experiments}
\label{sec.experi}
To validate our theoretical results, this section evaluates the empirical performance of our SCSC and Adam SCSC. 
We evaluate the empirical performance of SCSC and Adam SCSC in two tasks:  risk-averse portfolio management and sinusoidal regression for MAML. 
All experiments are run on a computer with Intel i9-9960x and NVIDIA Titan GPU. 

\subsection{Risk-averse portfolio management}
Given $d$ assets, let $\bbr_t\in\mathbb{R}^d$ denote the reward vector with $n$th entry representing the reward of $n$th asset observed at time slot $t$ over a total of $T$ slots. Portfolio management aims to find an investment $\bbtheta\in\mathbb{R}^d$ with $n$th entry representing the amount of investment or the split of the total investment allocated to the asset $n$. The optimal investment $\bbtheta^*$ is the one that solves the following problem
\begin{equation}\label{eq.sim-port}
    \max\limits_{\bbtheta\in\mathbb{R}^d}\ \frac{1}{T}\sum\limits_{t=1}^T \bbr_t^{\top} \bbtheta-\frac{1}{T}\sum\limits_{t=1}^T\Big(\bbr_t^{\top}\bbtheta-\frac{1}{T}\sum\limits_{j=1}^T \bbr_j^{\top}\bbtheta\Big)^2.
\end{equation}
 In this case, both random variables $\xi$ and $\phi$ in \eqref{opt0-2} are uniformly drawn from $\{\bbr_1,\cdots,\bbr_T\}$. 
If we define $g(\bbtheta; \bbr_j)=[\bbtheta, \bbr_j^{\top}\bbtheta]^{\top}\in\mathbb{R}^{d+1}$, and $\bby\in\mathbb{R}^{d+1}$ tracking $\mathbb{E}[g(\bbtheta; \bbr)]$, and define 
\begin{equation}
	f(\bby; \bbr_t)=\bby_{(d+1)}-(\bby_{(d+1)}-\bbr_t^{\top}\bby_{(1:d)})^2
\end{equation}
where $\bby_{(1:d)}$ and $\bby_{(d+1)}$ denote the first $d$ entries and the $(d+1)$th entry of $\bby$. In this case, problem \eqref{eq.sim-port} is an instance of stochastic composition problem \eqref{opt0-2}.

\begin{figure*}[t]
 \vspace*{-0.2cm}
\centering
   \includegraphics[width=0.33\textwidth]{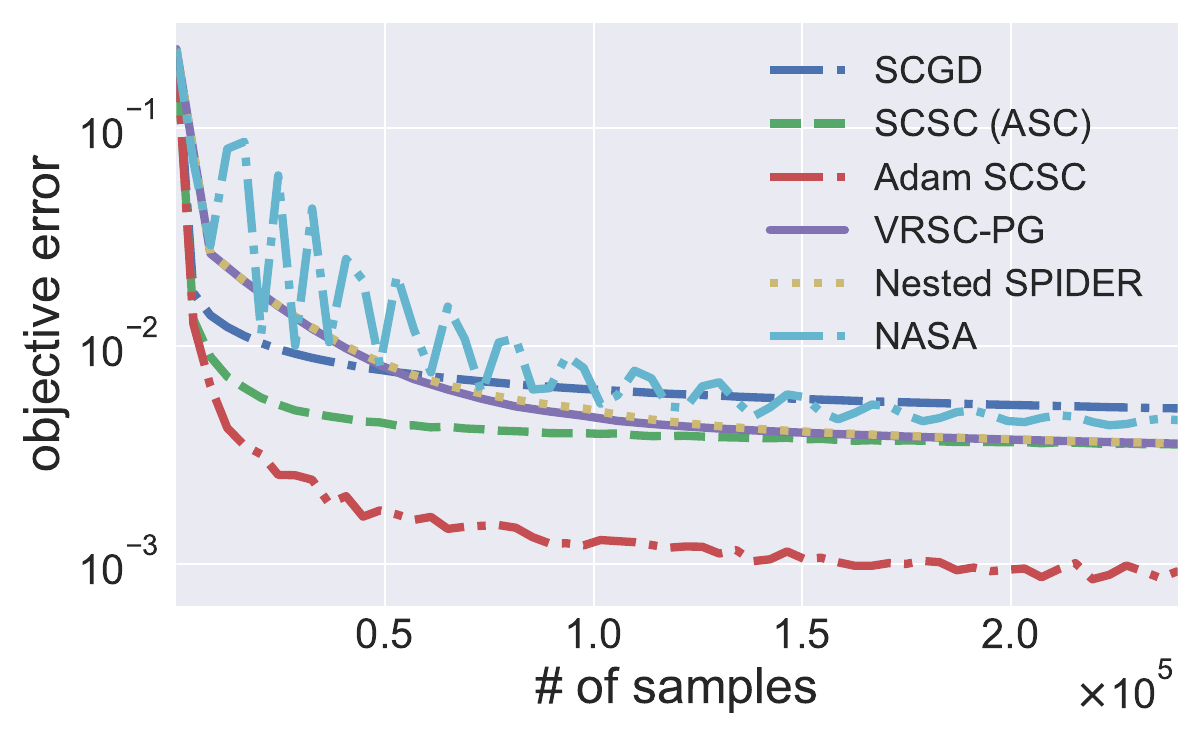} 
   \hspace{-0.15cm}
   \includegraphics[width=0.33\textwidth]{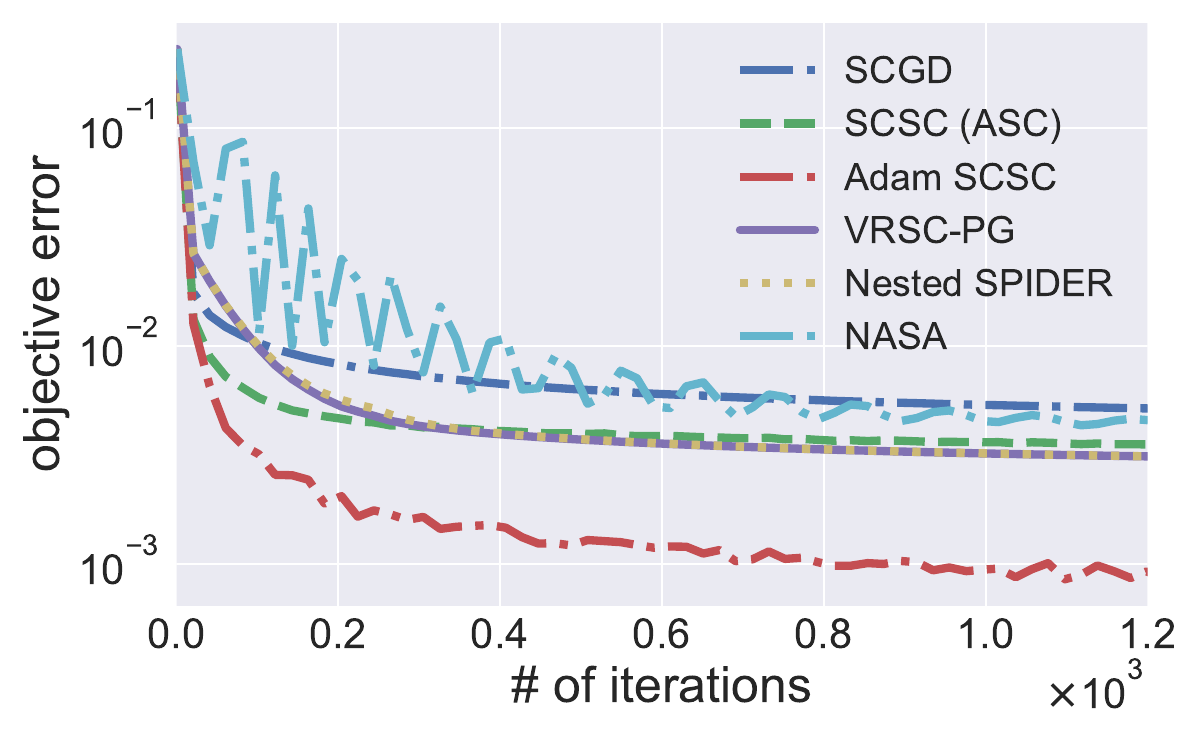} 
  \hspace{-0.15cm}
   \includegraphics[width=0.33\textwidth]{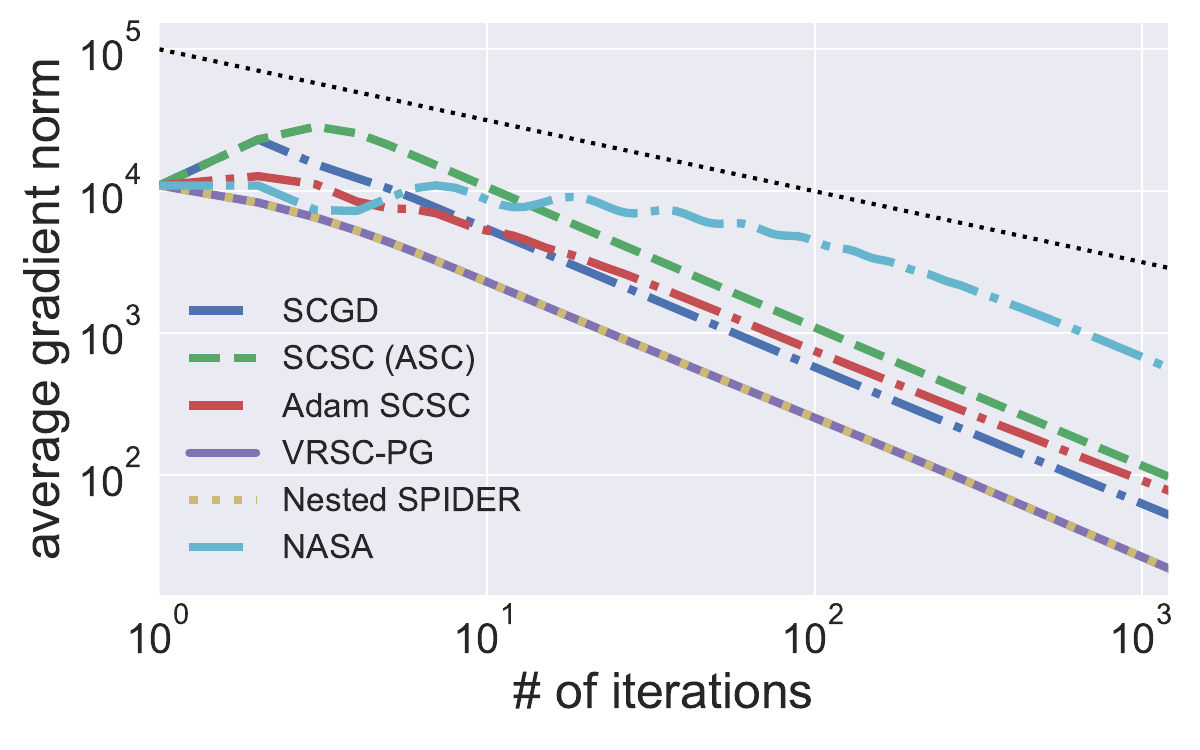} 
 \vspace*{-0.5cm}   
    \caption{\blue{Summary of results on the \emph{Industrial-49} dataset.}}
    \label{fig:portfolio1}
\end{figure*}

\begin{figure*}[t]
\centering
    \includegraphics[width=0.33\textwidth]{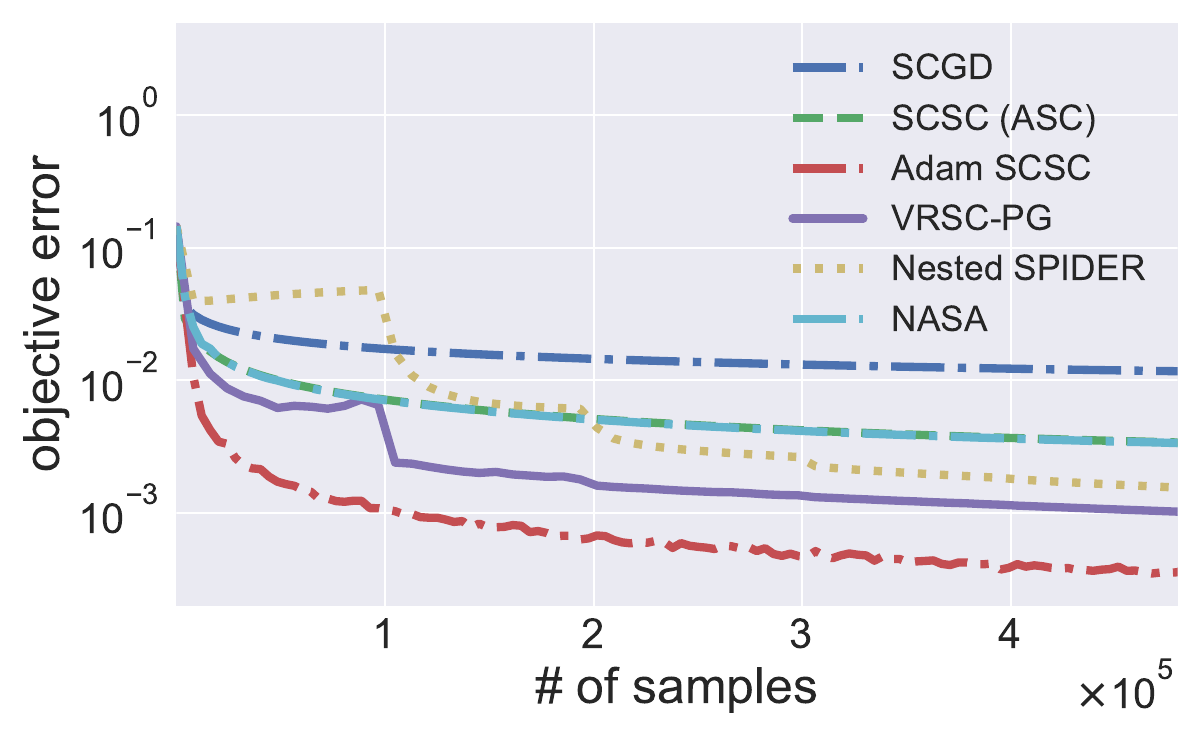}\hspace{-0.1cm}\includegraphics[width=0.33\textwidth]{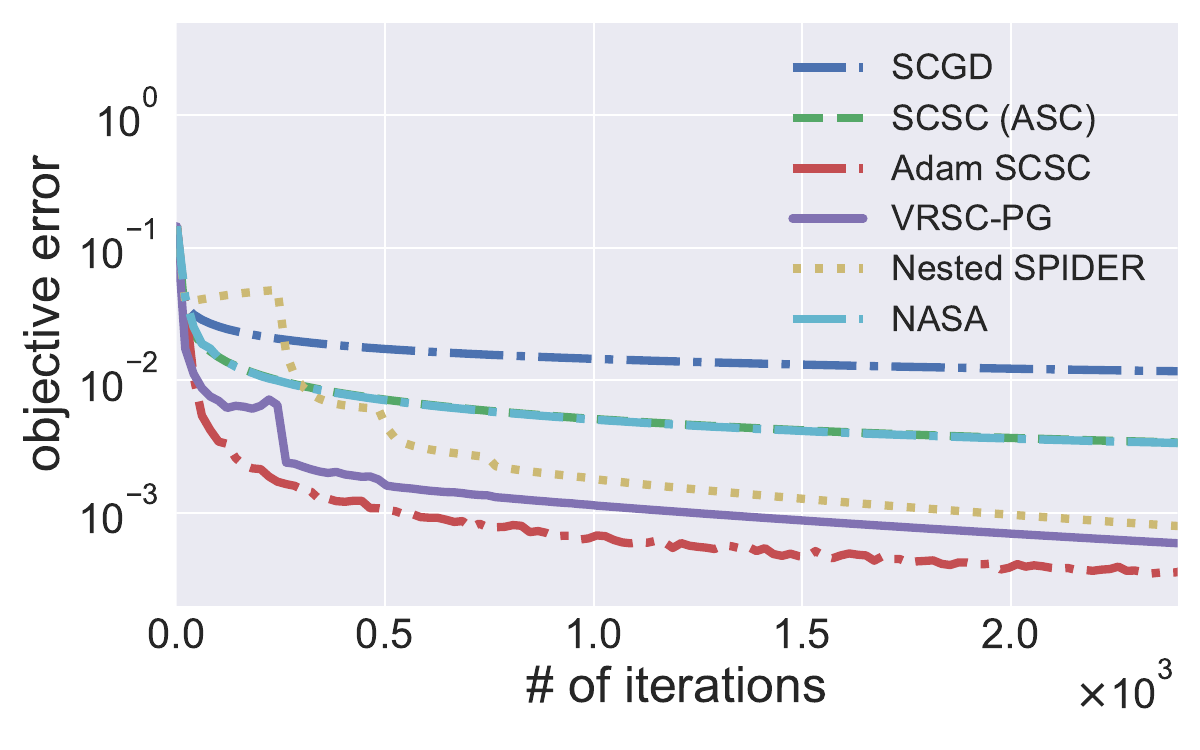} \hspace{-0.1cm}\includegraphics[width=0.33\textwidth]{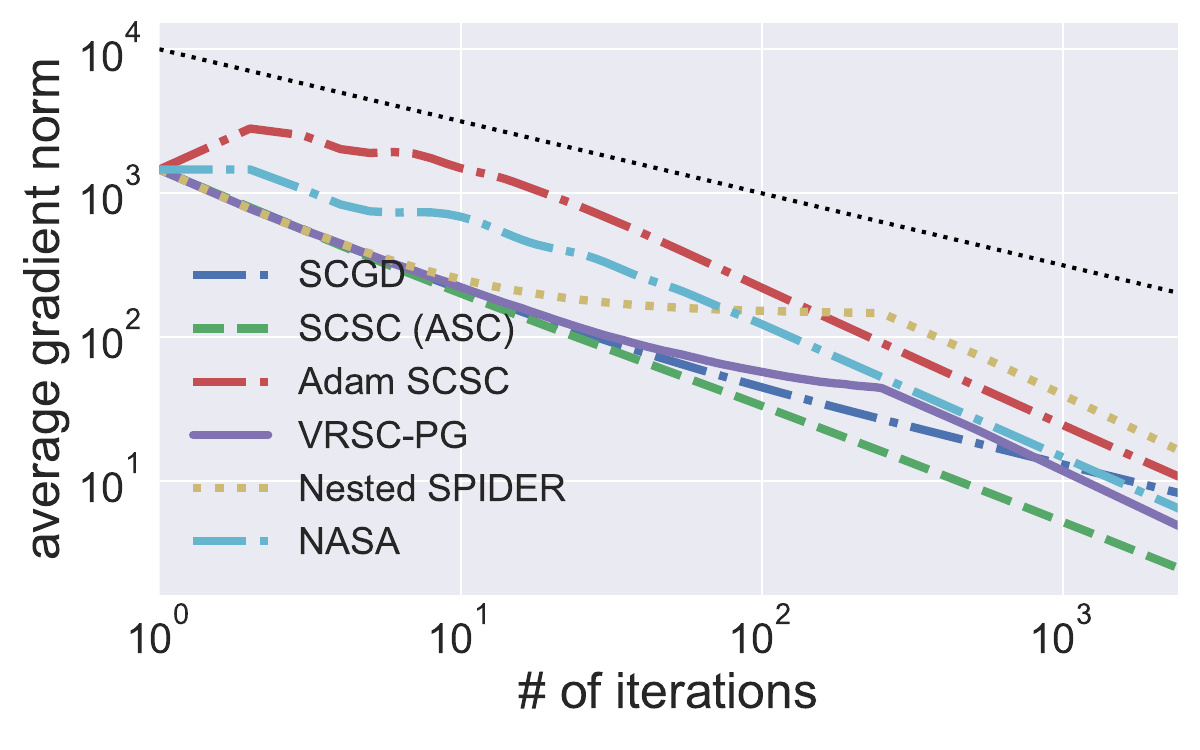} 
    \vspace*{-0.3cm}
    \caption{\blue{Summary of results on the \emph{100 Book-to-Market} dataset.}}
    \label{fig:portfolio2}
   \vspace*{-0.2cm}
\end{figure*}

\noindent\textbf{Benchmark algorithms.} 
We compare SCSC and Adam SCSC with SCGD\cite{wang2017mp}, VRSC-PG \cite{huo2018accelerated}, Nested SPIDER \cite{zhang2019nips} and \blue{the state-of-the-art NASA \cite{ghadimi2020jopt}.}  
For linear $g(\bbtheta; \bbr)$, SCSC is equivalent to the accelerated SCGD (ASC) \cite{wang2017jmlr}, and our SCSC and Adam SCSC under two different inner update rules \eqref{eq.SCSC-2} and \eqref{eq.SCSC-3} are also equivalent. Therefore, we only include SCSC with \eqref{eq.SCSC-3} in the simulation. 

\noindent\textbf{Hyperparameter tuning.} 
We tune the hyperparameters by first following the suggested order of stepsizes from the original papers and then using a grid search for the constant. 
For example, we choose {\small$\alpha_k=\alpha k^{-3/4}, \beta_k=k^{-1/2}$} for SCGD; {\small$\alpha_k=\alpha k^{-1/2}, \beta_k= k^{-1/2}$} for SCSC and Adam SCSC; the constant stepsize $\alpha$ for VRSC-PG and Nested SPIDER. 
The initial learning rate $\alpha$ is chosen from the searching grid  {\small$\{10^{-1}, 10^{-2}, 10^{-3}, 10^{-4}, 10^{-5}\}$} and optimized for each algorithm in terms of loss versus the number of iterations. Note that whenever the best performing hyperparameter lies in the boundary of the searching grid, we always extend the grid to make the final hyperparameter fall into the interior of the grid. 
For all the algorithms, we use the batch size 100 for inner and outer functions. 

Figures \ref{fig:portfolio1} and \ref{fig:portfolio2} show the test results averaged over 50 runs on two benchmark datasets: \emph{Industrial-49} and \emph{100 Book-to-Market}. 
The two datasets are downloaded from the Keneth R. French Data Library\footnote{http://mba.tuck.dartmouth.edu/pages/faculty/ken.french/data\_library.html}.
On both datasets, Adam SCSC achieves the best performance. SCSC outperforms several popular alternatives in \emph{Industrial-49} dataset, and performance very close to NASA in \emph{100 Book-to-Market} dataset.
\blue{To verify the empirical convergence rates of SCSC, Figures \ref{fig:portfolio1} and \ref{fig:portfolio2} have also shown the convergence of the average (squared) gradient norm in the log-log scale. By comparing the curves of SCSC and Adam SCSC with the dashed black line that indicates the theoretical ${\cal O}(k^{-\frac{1}{2}})$ rate in Theorems \ref{theorem1} and \ref{theorem4}, it is clear that the empirical convergence rates under SCSC and Adam SCSC are no worse than the worst-case theoretical rates.} 

\subsection{Sinusoidal regression for MAML}
For MAML, we consider the sinusoidal regression tasks as that in \cite{finn2017icml}. 
Each task in MAML is to regress from the input to the output of a sine wave 
$s(x;a,\varphi)=a\sin(x+\varphi)$, 
where the amplitude $a$ and phase $\varphi$ of the sinusoid vary across tasks.
We sample the amplitude $a$ uniformly from ${\cal U}([0.1, 5])$ and the phase $\varphi$ uniformly from ${\cal U}([0,2\pi])$. During training, datum $x$ is sampled uniformly from ${\cal U}([-5,5])$ and $s(x;a,\varphi)$ is observed. We use a neural network with 2 hidden layers and RELU activation functions, and use $\bbtheta$ for its weights $\bbtheta$ and $\hat{s}(x;\bbtheta)$ for its output. Using the mean square error $\mathbb{E}_x[\|\hat{s}(x;\bbtheta)-s(x;a,\varphi)\|^2]$,  we define
\begin{equation}
    F_m(\bbtheta)=\mathbb{E}_x[\|\hat{s}(x;\bbtheta)-s(x;a_m,\varphi_m)\|^2].
\end{equation}
In this case, to connect with \eqref{opt0-2}, both random variables $\xi$ and $\phi$ in \eqref{opt0-2} are uniformly drawn from ${\cal U}([-5,5])$. 
Let us define 
\begin{align}
    g(\bbtheta)\!=\![ g_1(\bbtheta)^{\top},\cdots, g_M(\bbtheta)^{\top}]^{\top}~{\rm with}~~g_m(\bbtheta)\!:=\!\bbtheta-\nabla F_m(\bbtheta)\nonumber
\end{align}
and define $\bby_m\in  \mathbb{R}^{d}$ to track $g_m(\bbtheta)$. With $\bby:=[\bby_1^{\top},\cdots,\bby_M^{\top}]^{\top}$, we define $f(\bby):=\frac{1}{M}\sum\limits_{m=1}^MF_m(\bby_m)$. 
Then MAML with sinusoidal regression satisfies the formulation \eqref{opt0-2}. 

\begin{figure}[t]
 \vspace*{-0.2cm}
\centering
    \includegraphics[width=0.4\textwidth]{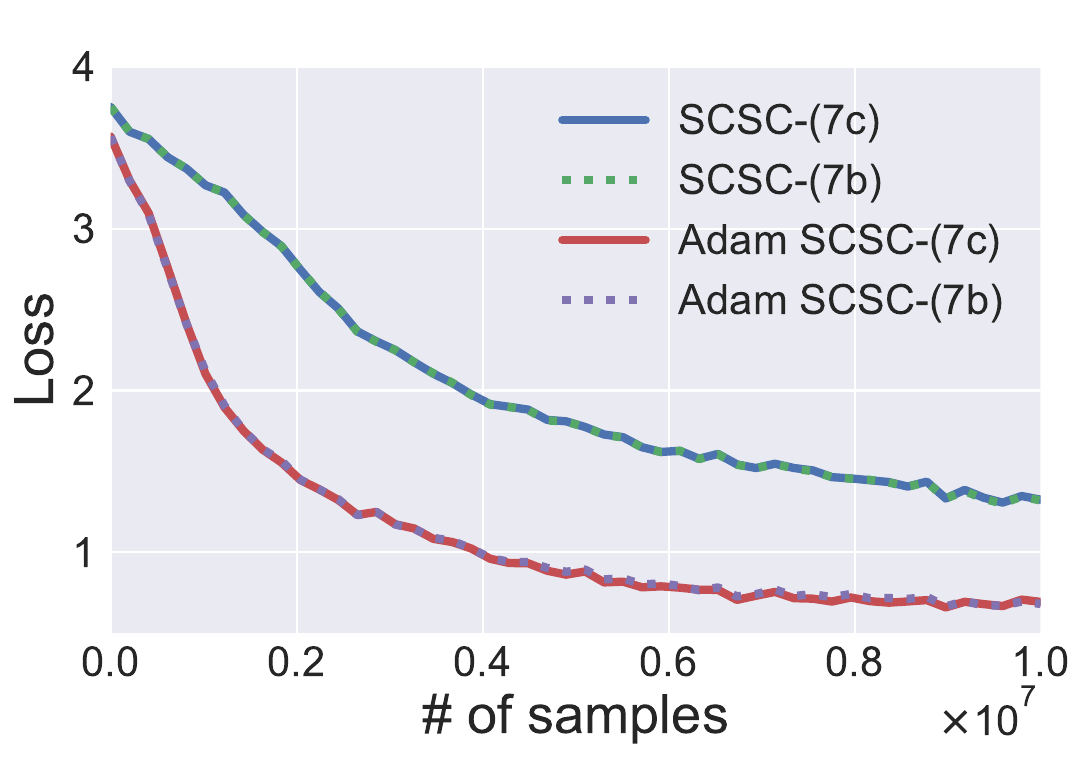}
    \hspace{-0.3cm}
    \includegraphics[width=0.42\textwidth]{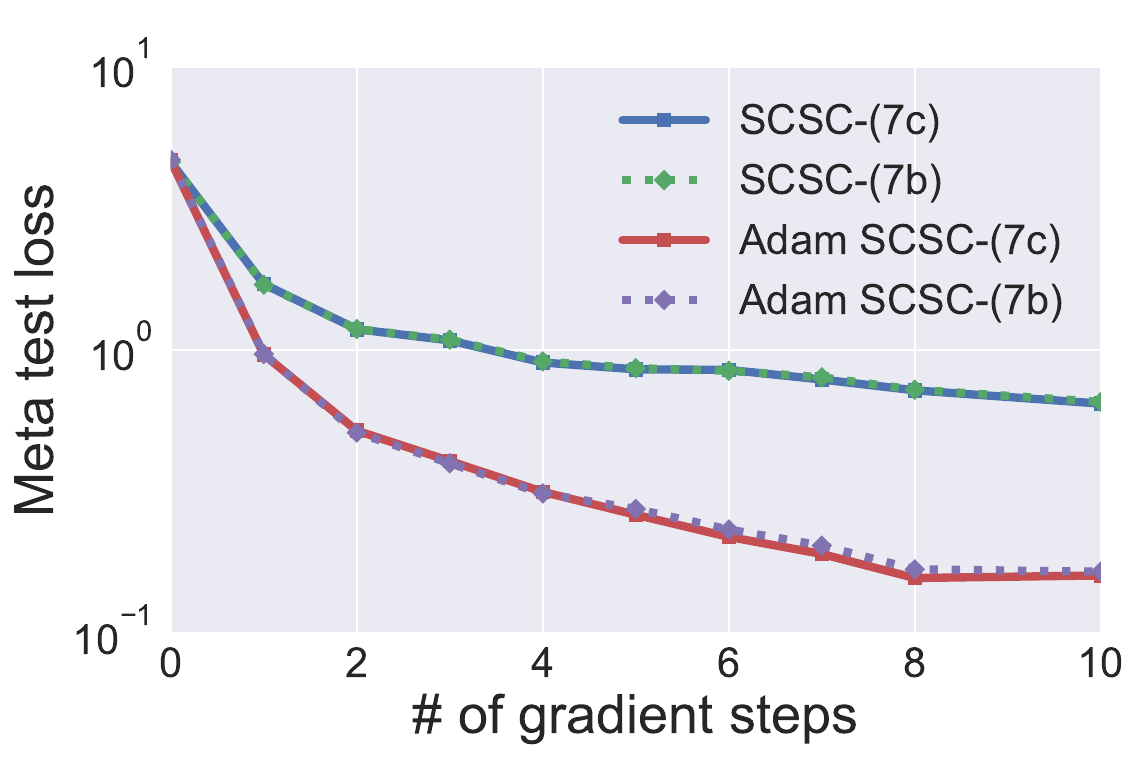}
  \vspace*{-0.5cm}
    \caption{{Comparison of two SCSC updates on the Sinewave regression task.}}
    \label{fig:maml2}
\end{figure}

\begin{figure}[t]
 \vspace*{-0.2cm}
\centering
    \includegraphics[width=0.4\textwidth]{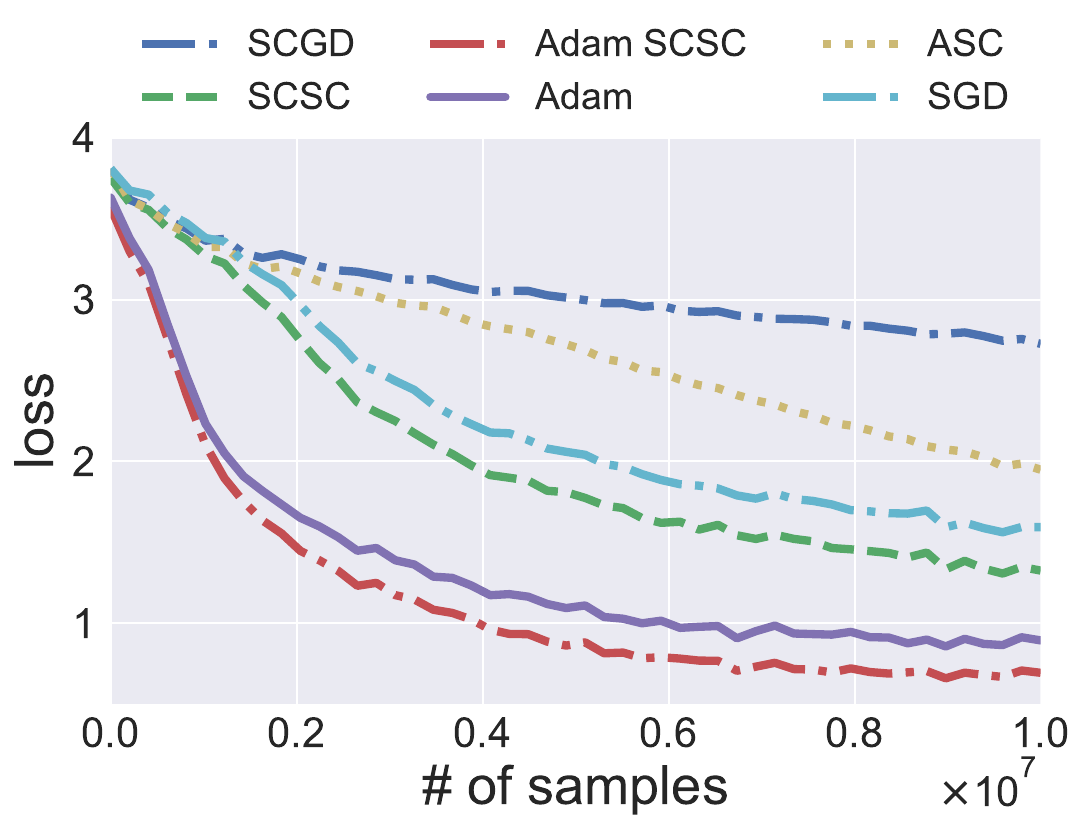}
    \includegraphics[width=0.42\textwidth]{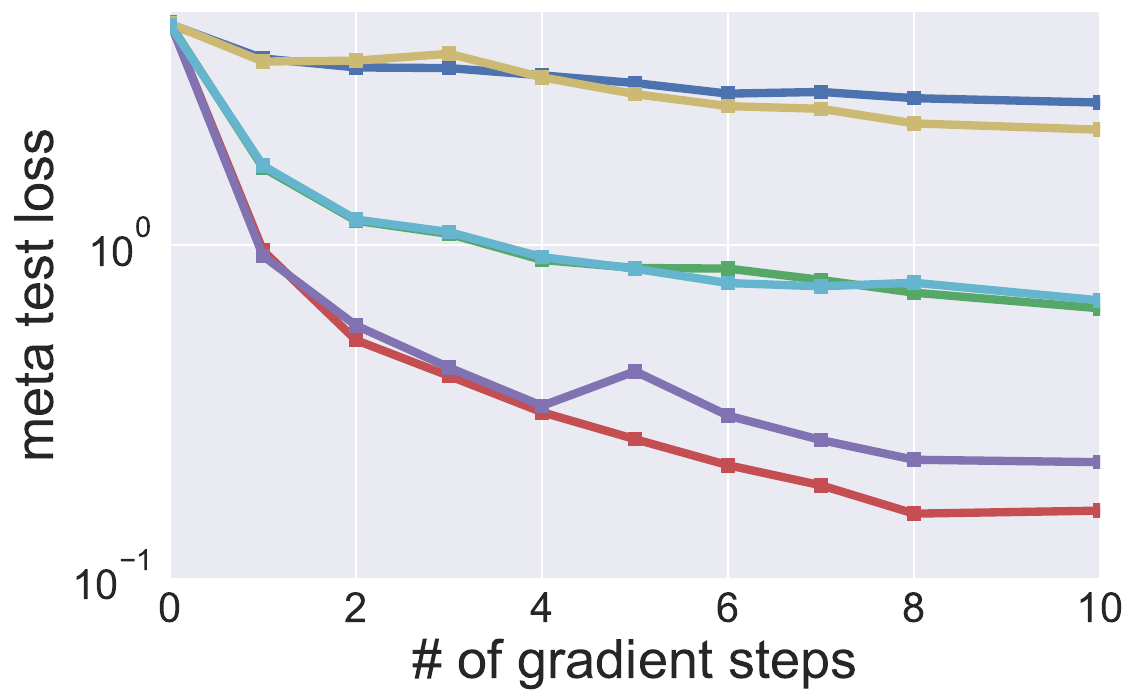}
     \vspace*{-0.2cm}
    \caption{\blue{Summary of results on the Sinewave regression task.}}
    \label{fig:maml}
        \vspace*{-0.2cm}
\end{figure}

\noindent\textbf{Benchmark algorithms.} 
In Figure \ref{fig:maml2}, we first compare the performance of SCSC and Adam SCSC under two different rules \eqref{eq.SCSC-2} and \eqref{eq.SCSC-3}. 
We then compare our SCSC and Adam SCSC with non-compositional stochastic optimization solver Adam and SGD (the common baseline for MAML), as well as compositional stochastic solver SCGD and ASC in Figure \ref{fig:maml}.

\noindent\textbf{Hyperparameter tuning.} 
We tune the hyperparameters by first following the suggested order of stepsizes from the original papers and then using a grid search for the constant. 
For SCSC and Adam SCSC, we use stepsizes $\alpha, \beta_k=0.8$. For Adam and SGD, we use $\alpha$. For SCGD and ASC, we use stepsizes $\alpha_k=\alpha k^{-3/4}, \beta_k=k^{-1/2}$ and $\alpha_k=\alpha k^{-5/9}$ and $\beta_k=k^{-4/9}$ as suggested in \cite{wang2017mp,wang2017jmlr}. 
The initial learning rate $\alpha$ is chosen from {\small$\{10^{-1}, 10^{-2}, 10^{-3}, 10^{-4}, 10^{-5}\}$} and optimized for each algorithm. 
During training, we fix $M=100$ and we sample 10 data from each task to evaluate the inner function $g(\bbtheta)$, and use another 10 data to evaluate $f(\bby)$. The MAML adaptation stepsize in \eqref{opt1} is $\alpha=0.01$. 

We compare the performance of SCSC and Adam SCSC under two different updates \eqref{eq.SCSC-2} and \eqref{eq.SCSC-3} in Figure \ref{fig:maml2}. Both \eqref{eq.SCSC-2} and \eqref{eq.SCSC-3} can guarantee that the new approach achieves the same convergence rate ${\cal O}(k^{-\frac{1}{2}})$, but \eqref{eq.SCSC-3} requires one more function evaluation than \eqref{eq.SCSC-2} at the old iterate $\bbtheta^{k-1}$. 
\blue{In terms of both the number of samples and number of gradients, the two update rules have very close performance, and the two lines are almost overlapping. Therefore, in the remaining tests, we will only plot SCSC with \eqref{eq.SCSC-3}.}  


\begin{figure}[t]
\centering
    \includegraphics[width=0.4\textwidth]{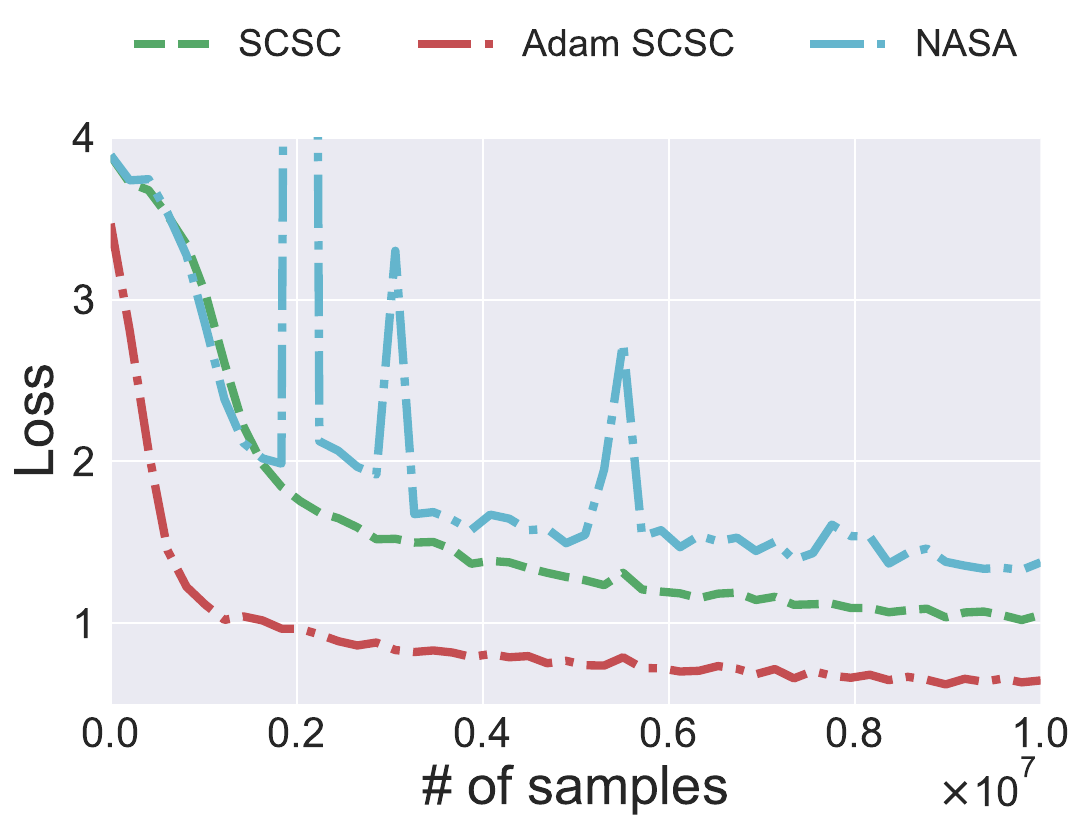}
    \includegraphics[width=0.42\textwidth]{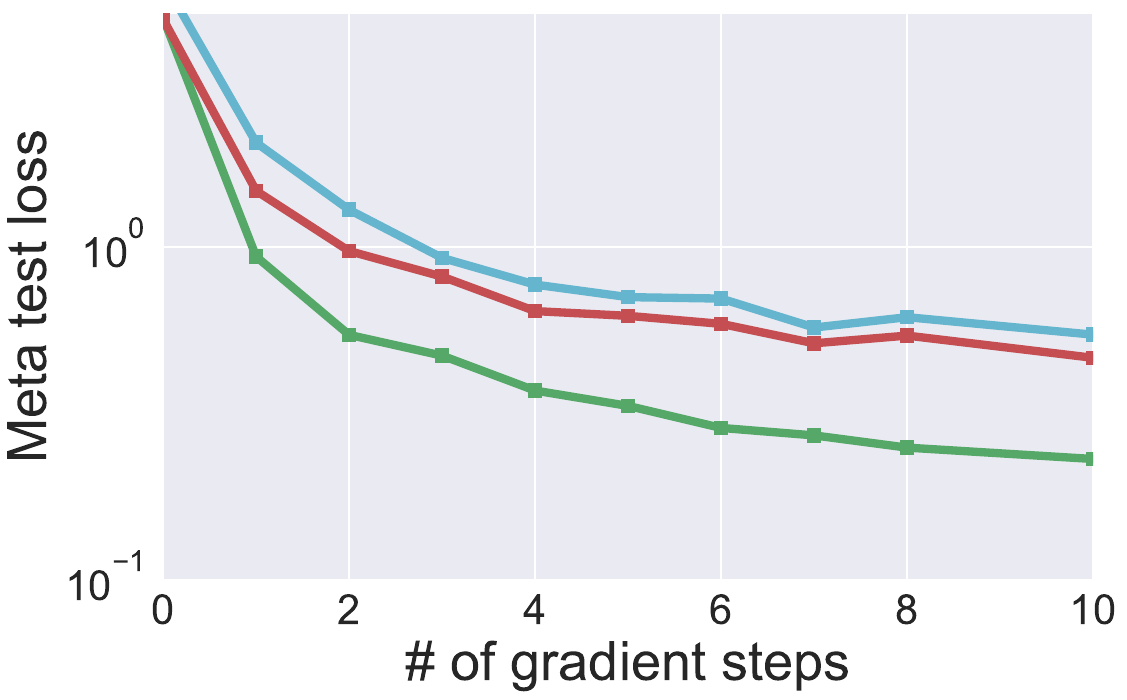}
     \vspace*{-0.2cm}
    \caption{\blue{Results on the Sinewave regression task (ELU activation functions).}}
    \label{fig:maml3}
        \vspace*{-0.2cm}
\end{figure}
 
In Figure \ref{fig:maml}, at each evaluation point of test loss, we sample 100 data to test the performance of each algorithm on these trained tasks. We also sample 100 unseen tasks to test the adaptation of the meta parameter learned on {\small$M=100$} tasks. For each unseen task, we start with the learned initialization and perform 10-step SGD. 
As shown in Figure \ref{fig:maml}, in terms of training loss, Adam SCSC again achieves the best performance, and SCSC outperforms the popular SCGD and ASC methods. In the meta test, while all algorithms reduce the test loss after several steps of adaptation, Adam SCSC achieves the fastest adaptation, and SCSC also has competitive performance.  
We have also simulated NASA \cite{ghadimi2020jopt} in this MAML task, but, partially due to the nonsmoothness of the objective function, we have observed the instability of the NASA algorithm. Hence, we compare SCSC and NASA under a smoothed MAML objective function using a smooth activation function ELU; see the results averaged over 5 random seeds in Figure \ref{fig:maml3}. In this smoothed setting, Adam SCSC again has the best performance. Albeit some oscillation, NASA has performance close to SCSC.


\section{Conclusions}
\label{sec.cons}

This paper presents a new method termed SCSC for solving the class of stochastic compositional optimization problems. 
SCSC runs in a single-time scale with a single loop, uses a fixed batch size. Remarkably, it converges at the same rate as the SGD method for non-compositional stochastic optimization. 
This is achieved by making a careful improvement to a popular stochastic compositional gradient method. 
Future research can be pursued in the following two dimensions: i) improving performance of SCSC by leveraging techniques such as decentralization, communication compression and robustness to asynchrony; and, ii) broadening the applicability of SCSC in other machine learning and signal processing applications.

%


\clearpage
%




\section{Proofs of main results}
In this section, we present the proofs of the theorems in Section \ref{sec.ic-ana} and the proofs of the multi-level case in the supplementary document. 

\subsection{Proof of Theorem \ref{theorem1}}

\subsubsection{Proof of Lemma \ref{lemma2} under update \eqref{eq.SCSC-3}}
From the update \eqref{eq.SCSC-3}, we have that
\begin{align}\label{eq.pflemma2-1}
\bby^{k+1}-g(\bbtheta^k)&=(1-\beta_k)(\bby^k-g(\bbtheta^{k-1})) +(1-\beta_k)(g(\bbtheta^{k-1})-g(\bbtheta^k))\nonumber\\
&\quad+\beta_k(g(\bbtheta^k;\phi^k)-g(\bbtheta^k)) +(1-\beta_k)(g(\bbtheta^k;\phi^k)-g(\bbtheta^{k-1};\phi^k))\nonumber\\
&=(1-\beta_k)(\bby^k-g(\bbtheta^{k-1}))+(1-\beta_k)T_1 +\beta_k T_2+(1-\beta_k)T_3
\end{align}
where we define the three terms as $T_1:=g(\bbtheta^{k-1})-g(\bbtheta^k)$, $T_2:=g(\bbtheta^k;\phi^k)-g(\bbtheta^k)$, and $T_3:=g(\bbtheta^k;\phi^k)-g(\bbtheta^{k-1};\phi^k)$.

\blue{Conditioned on ${\cal F}^k$, taking expectation on the both sides of \eqref{eq.pflemma2-1}, we have 
\begin{align*}\label{eq.app-a-1}
\,&\EE[\|\bby^{k+1}-g(\bbtheta^k)\|^2|\mathcal F^k]\\
=\,&(1-\beta_k)^2\EE[\|(\bby^k-g(\bbtheta^{k-1}))\|^2|\mathcal F^k] +\EE\left[\|(1-\beta_k)T_1+\beta_k T_2+(1-\beta_k)T_3\|^2|\mathcal F^k\right]\\
&+2\Big\langle (1-\beta_k)(\bby^k-g(\bbtheta^{k-1})),  \EE\Big[(1-\beta_k)T_1+\beta_k T_2+(1-\beta_k)T_3|\mathcal F^k\Big]\Big\rangle.\numberthis
\end{align*}

For the second term in the RHS of \eqref{eq.app-a-1}, using the Young's inequality, we have 
\begin{align*}\label{eq.app-a-1-2}
	&\EE\left[\|(1-\beta_k)T_1+\beta_k T_2+(1-\beta_k)T_3\|^2|\mathcal F^k\right]\\
\leq &  2(1-\beta_k)^2\EE\left[ \|T_1+T_3\|^2|\mathcal F^k\right]+2 \beta_k^2\EE\left[\|T_2\|^2|\mathcal F^k\right]\\
= & 2(1-\beta_k)^2\EE[\|T_1\|^2\mid{\cal F}^k]+2\beta_k^2\EE[\|T_2\|^2\mid{\cal F}^k]\\
&+4 (1-\beta_k)^2\left\langle T_1,\EE[T_3\mid{\cal F}^k]\right\rangle+2(1-\beta_k)^2\EE\left[\|T_3\|^2|\mathcal F^k\right].\numberthis
\end{align*}

Using Assumptions 3 and 4, from \eqref{eq.app-a-1-2}, we have
\begin{align*}\label{eq.app-a-2}
	&\EE\left[\|(1-\beta_k)T_1+\beta_k T_2+(1-\beta_k)T_3\|^2|\mathcal F^k\right]\\
\leq &  -2(1-\beta_k)^2\EE\left[\|g(\bbtheta^k)-g(\bbtheta^{k-1})\|^2|\mathcal F^k\right] +2(1-\beta_k)^2\EE\left[\|g(\bbtheta^k;\phi^k)-g(\bbtheta^{k-1};\phi^k)\|^2|\mathcal F^k\right]+2\beta_k^2V_g^2\\
\leq &  2(1-\beta_k)^2C_g^2\|\bbtheta^k-\bbtheta^{k-1}\|^2+2\beta_k^2V_g^2\numberthis
\end{align*}
where the second inequality follows from Assumption 2. }

For the third term in the RHS of \eqref{eq.app-a-1}, conditioned on ${\cal F}^k$, taking expectation over $\phi^k$, we have
\begin{align}\label{eq.app-a-3}
\EE\left[(1-\beta_k)T_1+\beta_k T_2+(1-\beta_k)T_3|\mathcal F^k\right]=\mathbf{0}.
\end{align}
Plugging \eqref{eq.app-a-2} and \eqref{eq.app-a-3} into \eqref{eq.app-a-1}, the proof is complete.

\subsubsection{Proof of Lemma \ref{lemma2} under update \eqref{eq.SCSC-2}}
 \begin{lemma}[Tracking error under \eqref{eq.SCSC-2}]\label{lemma1-2}
Suppose that Assumptions 1-4 hold, and $\bby^{k+1}$ is generated by running iteration \eqref{eq.SCSC} given $\bbtheta^k$. Then the variance of $\bby^{k+1}$ satisfies 
\begin{align}\label{eq.lemma1-2}
 \EE\left[\|g(\bbtheta^k)-\bby^{k+1}\|^2\mid{\cal F}^k\right] 
\leq & (1-\beta_k)\|\bby^{k-1}-g(\bbtheta^k)\|^2+4(1-\beta_k)^2C_g^2\|\bbtheta^k-\bbtheta^{k-1}\|^2
\nonumber\\
&+2\beta_k^2V_g^2+\frac{(1-\beta_k)^2L_g^2}{4\beta_k}\|\bbtheta^k-\bbtheta^{k-1}\|^4.
\end{align}
\end{lemma}
\begin{proof}
For \eqref{eq.SCSC-2}, using the fact that $\nabla g(\bbtheta)$ is $L_g$-Lipschitz continuous in Assumption 1, we have
\begin{align}\label{eq.pflemma2-2-1}
 \bby^{k+1}-g(\bbtheta^k) 
 = & (1-\beta_k)(\bby^k-g(\bbtheta^{k-1}))+(1-\beta_k)(g(\bbtheta^k)-g(\bbtheta^{k-1}))\nonumber\\
& +\beta_k(g(\bbtheta^k;\phi^k)-g(\bbtheta^k)) +(1-\beta_k)\nabla g(\bbtheta^{k-1};\phi^k)(\bbtheta^k-\bbtheta^{k-1})\nonumber\\
=& (1-\beta_k)(\bby^k-g(\bbtheta^{k-1})+(1-\beta_k)T_1+\beta_k T_2+(1-\beta_k)T_3 
\end{align}
where we define the terms as $T_1:=g(\bbtheta^{k-1})-g(\bbtheta^k)$, $T_2:=g(\bbtheta^k;\phi^k)-g(\bbtheta^k)$, and $T_3:=\nabla g(\bbtheta^{k-1};\phi^k)(\bbtheta^k-\bbtheta^{k-1})$.

Therefore, conditioned on ${\cal F}^k$, taking expectation on both sides of \eqref{eq.pflemma2-2-1} over $\phi^k$, we have
 \begin{align*}\label{eq.app-a-4}
    &\, \EE[\|\bby^{k+1}-g(\bbtheta^k)\|^2\mid{\cal F}^k]\\
   =&\, (1-\beta_k)^2\|\bby^k-g(\bbtheta^{k-1})\|^2+  2\Big\langle\!(1-\beta_k)(\bby^k\!-\!g(\bbtheta^{k-1})),  \EE\big[(1-\beta_k)T_1\!+\!\beta_kT_2\!+\!(1-\beta_k)T_3\!\mid\!{\cal F}^k\big]\Big\rangle\\
      &+\EE[\|(1-\beta_k)T_1+\beta_kT_2+(1-\beta_k)T_3\|^2\mid{\cal F}^k].\numberthis
\end{align*}

 For the second term in the RHS of \eqref{eq.app-a-4}, using the Cauchy-Schwartz inequality, we have 
 \begin{align*}
&\left\langle\!(1-\beta_k)(\bby^k-g(\bbtheta^{k-1})), \EE\big[(1-\beta_k)T_1\!+\!\beta_kT_2\!+\!(1-\beta_k)T_3\!\mid\!{\cal F}^k\big]\right\rangle\\
&\leq (1-\beta_k) \big\|\bby^k-g(\bbtheta^{k-1})\big\| \Big\|\EE[(1-\beta_k)T_1+\beta_k T_2+(1-\beta_k)T_3\mid{\cal F}^k]\Big\|.
\end{align*} 

Conditioned on ${\cal F}^k$, taking expectation over $\phi^k$, we have
{\small\begin{align}\label{eqn:1}
    &\, \left\|\EE\left[(1-\beta_k)T_1+\beta_k T_2+(1-\beta_k)T_3\mid{\cal F}^k\right]\right\|\nonumber\\
    =&\, (1-\beta_k)\left\|g(\bbtheta^{k-1})-g(\bbtheta^k)+\nabla g(\bbtheta^{k-1})(\bbtheta^k-\bbtheta^{k-1})\right\|\nonumber\\
    =&\, (1-\beta_k)\Big\|\int_0^1-\nabla g(\bbtheta^{k-1}+t(\bbtheta^k-\bbtheta^{k-1}))(\bbtheta^k-\bbtheta^{k-1})dt +\nabla g(\bbtheta^{k-1})(\bbtheta^k-\bbtheta^{k-1})\Big\|\nonumber\\
    \leq& (1-\beta_k)\!\!\int_0^1\!\left\|\nabla g(\bbtheta^{k-1})\!-\!\nabla g(\bbtheta^{k-1}\!\!+\!t(\bbtheta^k-\bbtheta^{k-1}))\right\|\!\|\bbtheta^k-\bbtheta^{k-1}\|dt\nonumber\\
    \leq&\, (1-\beta_k)\int_0^1L_gt\|\bbtheta^{k}-\bbtheta^{k-1}\|^2 = \frac{(1-\beta_k)L_g}{2}\|\bbtheta^k-\bbtheta^{k-1}\|^2.
\end{align} 

 For the third term in the RHS of \eqref{eq.app-a-4}, following the steps of \eqref{eq.app-a-1-2}-\eqref{eq.app-a-2}, we have
 \begin{align*}\label{eq.app-a-5}
 \EE[\|(1-\beta_k)T_1+\beta_kT_2+(1-\beta_k)T_3\|^2\mid{\cal F}^k] 
 \leq  4(1-\beta_k)^2C_g^2\|\bbtheta^k-\bbtheta^{k-1}\|^2+2\beta_k^2V_g^2.\numberthis
 \end{align*}
 
 Plugging \eqref{eqn:1} and \eqref{eq.app-a-5} into \eqref{eq.app-a-4}, we have
 \begin{align*}
    &\, \EE[\|\bby^{k+1}-g(\bbtheta^k)\|^2\mid{\cal F}^k]\\
\leq &\, (1-\beta_k)^2\|\bby^k-g(\bbtheta^{k-1})\|^2+4(1-\beta_k)^2C_g^2\|\bbtheta^k-\bbtheta^{k-1}\|^2\\
   &+2\beta_k^2V_g^2+(1-\beta_k)^2L_g\|\bby^k-g(\bbtheta^{k-1})\|\|\bbtheta^k-\bbtheta^{k-1}\|^2\\
 \stackrel{(a)}{\leq}&\, (1-\beta_k)^2\|\bby^k-g(\bbtheta^{k-1})\|^2+4(1-\beta_k)^2C_g^2\|\bbtheta^k-\bbtheta^{k-1}\|^2+2\beta_k^2V_g^2\\
   &\, +(1-\beta_k)^2\beta_k\|\bby^k-g(\bbtheta^{k-1})\|^2+\frac{(1-\beta_k)^2L_g^2}{4\beta_k}\|\bbtheta^k-\bbtheta^{k-1}\|^4\\
\stackrel{(b)}{\leq}&\, (1-\beta_k)^2(1+\beta_k)\|\bby^k-g(\bbtheta^{k-1})\|^2+4(1-\beta_k)^2C_g^2\|\bbtheta^k-\bbtheta^{k-1}\|^2 +2\beta_k^2V_g^2+\frac{(1-\beta_k)^2L_g^2}{4\beta_k}\|\bbtheta^k-\bbtheta^{k-1}\|^4
\end{align*}}
where (a) uses the Young's inequality, and (b) follows from $(1-\beta_k)^2(1+\beta_k)\leq 1-\beta_k$.
Hence, the proof is complete.  
\end{proof}

Compared with the tracking variance in Lemma \ref{lemma2} under \eqref{eq.SCSC-3}, Lemma \ref{lemma1-2} under \eqref{eq.SCSC-2} has an additional term $\frac{(1-\beta_k)^2L_g^2}{4\beta_k}\|\bbtheta^k-\bbtheta^{k-1}\|^4$.  
In this case, under a stronger version of Assumption 2' (e.g., bounded fourth moments), this term is {\small${\cal O}\left(\alpha_k^4/\beta_k\right)$}, which will be dominated by second and the third terms in the RHS of \eqref{eq.lemma1-2} since both of them are {\small${\cal O}\left(\alpha_k^2\right)$}. 

Once we have established this, the remaining proof of SCSC with \eqref{eq.SCSC-2} follows the same line as that of SCSC with \eqref{eq.SCSC-3}. For brevity, we only present the proof under Lemma \ref{lemma2}, and that under Lemma \ref{lemma1-2} follows similarly.

\noindent\textbf{Assumption 2'.}
\emph{The stochastic gradients of $f$ and $g$ are bounded in expectation, that is {\small$\mathbb{E}\left[\|\nabla g(\bbtheta;\phi)\|^4\right]\leq C_g^4$ and $\mathbb{E}\left[\|\nabla f(\bby;\xi)\|^4\right]\leq C_f^4$}.}

\subsubsection{Remaining proof}
Using the $L$-smoothness of $F(\bbtheta)$ in \eqref{eq.smooth_const}, we have
 \begin{align*}
  F(\bbtheta^{k+1})-F(\bbtheta^k) 
    &\leq \dotp{\nabla F(\bbtheta^k),\bbtheta^{k+1}-\bbtheta^k}+\frac{L}{2}\|\bbtheta^{k+1}-\bbtheta^k\|^2\\
    &= -\alpha_k\dotp{\nabla F(\bbtheta^k),\nabla g(\bbtheta^k;\phi^k)\nabla f(\bby^{k+1};\xi^k)}+\frac{L}{2}\|\bbtheta^{k+1}-\bbtheta^k\|^2\\
    &= -\alpha_k\|\nabla F(\bbtheta^k)\|^2+\frac{L}{2}\|\bbtheta^{k+1}-\bbtheta^k\|^2\\
    &\ \ \ +\alpha_k\dotp{\nabla F(\bbtheta^k),\nabla g(\bbtheta^k)\nabla f(g(\bbtheta^k))-\nabla g(\bbtheta^k;\phi^k)\nabla f(\bby^{k+1};\xi^k)}
\end{align*} 
where the first equality follows from \eqref{eq.SCSC-1} and the last equality uses $\nabla F(\bbtheta^k)=\nabla g(\bbtheta^k)\nabla f(g(\bbtheta^k))$.

Conditioned on $\mathcal F^k$, taking expectation over $\phi^k$ and $\xi^k$ on both sides, we have
{\small\begin{align}\label{eq.pf37}
    &\EE\left[F(\bbtheta^{k+1})|\mathcal F^k\right]-F(\bbtheta^k)\nonumber\\
 \stackrel{(a)}{\leq} & -\alpha_k\|\nabla F(\bbtheta^k)\|^2+\frac{L}{2}\EE\left[\|\bbtheta^{k+1}-\bbtheta^k\|^2|\mathcal F^k\right]\nonumber\\
    &\!+\!\alpha_k\EE\left[\!\dotp{\nabla F(\bbtheta^k),\nabla g(\bbtheta^k;\phi^k)(\nabla f(g(\bbtheta^k);\xi^k)\!-\!\nabla f(\bby^{k+1};\xi^k)}|\mathcal F^k\right]\nonumber\\
    \stackrel{(b)}{\leq}& -\alpha_k\|\nabla F(\bbtheta^k)\|^2+\frac{L}{2}\EE[\|\bbtheta^{k+1}-\bbtheta^k\|^2|\mathcal F^k]\nonumber\\
    & +\alpha_k\left\|\nabla F(\bbtheta^k)\right\|\EE\left[\|\nabla g(\bbtheta^k;\phi^k)\|^2|\mathcal F^k\right]^{\frac{1}{2}}\EE\!\left[\|\nabla f(g(\bbtheta^k);\xi^k)\!-\!\nabla f(\bby^{k+1};\xi^k)\|^2|\mathcal F^k\right]^{\frac{1}{2}}
\end{align}}
\hspace{-0.4cm}where (a) uses $\EE[\nabla g(\bbtheta^k;\phi^k)\nabla f(g(\bbtheta^k);\xi^k)|\mathcal F^k]=\nabla g(\bbtheta^k)\nabla f(g(\bbtheta^k))$ in Assumption 3, and (b) uses the Cauchy-Schwartz inequality. 

Further expanding the RHS of \eqref{eq.pf37}, we have
\begin{align}
    &\frac{L}{2}\EE[\|\bbtheta^{k+1}-\bbtheta^k\|^2|\mathcal F^k]\leq \frac{L}{2}C_g^2C_f^2\alpha_k^2
\end{align}
which follows from Assumption 2. 

And we have
 \begin{align*}    
    &\alpha_k\left\|\nabla F(\bbtheta^k)\right\|\EE\left[\|\nabla g(\bbtheta^k;\phi^k)\|^2|\mathcal F^k\right]^{\frac{1}{2}} \EE\left[\|\nabla f(g(\bbtheta^k);\xi^k)\!-\!\nabla f(\bby^{k+1};\xi^k)\|^2|\mathcal F^k\right]^{\frac{1}{2}}\\
    \stackrel{(c)}{\leq} & \alpha_kC_gL_f\|\nabla F(\bbtheta^k)\|\EE\left[\|g(\bbtheta^k)-\bby^{k+1}\|^2|\mathcal F^k\right]^{\frac{1}{2}}\\
    \stackrel{(d)}{\leq} & \frac{\alpha_k^2}{4\beta_k}C_g^2L_f^2\|\nabla F(\bbtheta^k)\|^2+\beta_k\EE\left[\|g(\bbtheta^k)-\bby^{k+1}\|^2|\mathcal F^k\right] 
\end{align*} 
where (c) uses Assumptions 1 and 2; and (d) uses the Young's inequality. 

Therefore, we have
\begin{align}
&\EE\left[F(\bbtheta^{k+1})|\mathcal F^k\right]-F(\bbtheta^k)\nonumber\\
	\leq &\, -\alpha_k\left(1-\frac{\alpha_k}{4\beta_k}C_g^2L_f^2\right)\|\nabla F(\bbtheta^k)\|^2 +\beta_k\EE\left[\|g(\bbtheta^k)-\bby^{k+1}\|^2|\mathcal F^k\right]+\frac{L}{2}C_g^2C_f^2\alpha_k^2.
\end{align}

Then with the definition of Lyapunov function in \eqref{eq.Lyap}, it follows that
\begin{align}\label{eq.thm1-1}
&    \EE[{\cal V}^{k+1}|\mathcal F^k]-{\cal V}^k\nonumber\\
    &\leq-\alpha_k\left(1-\frac{\alpha_k}{4\beta_k}C_g^2L_f^2\right)\|\nabla F(\bbtheta^k)\|^2+\frac{L}{2}C_g^2C_f^2\alpha_k^2 \\
    &\ \ \ \ +(1+\beta_k)\EE\left[\|g(\bbtheta^k)-\bby^{k+1}\|^2|\mathcal F^k\right]-\|g(\bbtheta^{k-1})-\bby^k\|^2\nonumber\\
    &\stackrel{(a)}{\leq}  -\alpha_k\left(1-\frac{\alpha_k}{4\beta_k}C_g^2L_f^2\right)\|\nabla F(\bbtheta^k)\|^2+\frac{L}{2}C_g^2C_f^2\alpha_k^2+2(1+\beta_k)\beta_k^2V_g^2 \!\nonumber\\
    &\ \ \ \ +\!\left((1+\beta_k)(1-\beta_k)^2\!-\!1\right)\|g(\bbtheta^{k-1})-\bby^k\|^2  +4(1+\beta_k)(1-\beta_k)^2C_g^4C_f^2\alpha_k^2 \nonumber\\
    &\stackrel{(b)}{\leq} -\alpha_k\left(1-\frac{\alpha_k}{4\beta_k}C_g^2L_f^2\right)\|\nabla F(\bbtheta^k)\|^2+\frac{L}{2}C_g^2C_f^2\alpha_k^2+2(1+\beta_k)\beta_k^2V_g^2+4C_g^4C_f^2\alpha_k^2 \nonumber
\end{align}
where (a) follows from Lemma \ref{lemma2}, and (b) uses that $(1+\beta_k)(1-\beta_k)^2=(1-\beta_k^2)(1-\beta_k)\leq 1$. 

Select $\alpha_k=\frac{2\beta_k}{C_g^2L_f^2}$ so that $1-\frac{\alpha_k}{4\beta_k}C_g^2L_f^2=\frac{1}{2}$, and define (with $\beta_k\in(0,1)$)
\begin{equation}
\small
    B_1:=\frac{L}{2}C_g^2C_f^2+4V_g^2+4C_g^4C_f^2\geq \frac{L}{2}C_g^2C_f^2+2(1+\beta_k)V_g^2+4C_g^4C_f^2.
\end{equation}
Taking expectation over $\mathcal F^k$ on both sides of \eqref{eq.thm1-1}, then it follows that
\begin{align}\label{eq.thm1-2}
    \EE[{\cal V}^{k+1}]\leq \EE[{\cal V}^k]-\frac{\alpha_k}{2}\EE[\|\nabla F(\bbtheta^k)\|^2]+B_1\alpha_k^2.
\end{align}
Rearranging terms, we have     
 \begin{align*}   
    \frac{\sum_{k=0}^K\alpha_k\EE[\|\nabla F(\bbtheta^k)\|^2]}{\sum_{k=0}^K\alpha_k}\leq\frac{2{\cal V}^0+2 B_1\sum_{k=0}^K\alpha_k^2}{\sum_{k=0}^K\alpha_k}.
\end{align*}
Choosing the stepsize as $\alpha_k=\frac{1}{\sqrt{K}}$ completes the proof.

\subsection{Proof of Theorem \ref{theorem4}}
\subsubsection{Supporting lemmas}
We first present the essential lemmas that will lead to Theorem \ref{theorem4}. 
\begin{lemma}\label{lemma3}
Under Assumption 5, the parameters $\{\bbh^k, \hat{\bbv}^k\}$ of Adam SCSC in Algorithm \ref{alg:adascg} satisfy
\begin{equation}
\|\bbh^k\|\leq C_gC_f,~~~\forall k;~~~\hat{v}_i^k\leq C_g^2C_f^2, ~~~\forall k, i.
\end{equation}
\end{lemma}
\begin{proof}
Using Assumption 5, it follows that $\|\bm\nabla^k\|=\|\nabla g(\bbtheta^k;\phi^k) \nabla f(\bby^{k+1};\xi^k)\|\leq C_gC_f$.  
Therefore, from the update \eqref{eq.adaSCGD-1}, we have
\begin{align*}
\|\bbh^{k+1}\|\!\leq\! \eta_1\|\bbh^{k}\|\!+\!(1-\eta_1)\|\bm\nabla^k\|\leq \eta_1\|\bbh^{k}\|+(1-\eta_1) C_gC_f.
\end{align*}
Since $\|\bbh^{1}\|\leq C_gC_f$, by induction, we have $\|\bbh^{k+1}\|\leq C_gC_f$. 

Similarly, from the update \eqref{eq.adaSCGD-2}, we have
\begin{align*}
\hat v_i^{k+1}&\leq\max\{\hat v_i^k,\eta_2 \hat v_i^k+(1-\eta_2)(\nabla_i^k)^2\}\\
&\leq \max\{\hat v_i^k,\eta_2 \hat v_i^k+(1-\eta_2)C_g^2C_f^2\}. 
\end{align*}
Since $v_i^1=\hat v_i^1\leq C_g^2C_f^2$, by induction, $\hat v_i^{k+1}\leq C_g^2C_f^2$. 
\end{proof}

\begin{lemma}\label{lemme4}
Under Assumption 5, the iterates $\{\bbtheta^k\}$ of Adam SCSC in Algorithm \ref{alg:adascg} satisfy
\begin{equation}
    \left\|\bbtheta^{k+1}-\bbtheta^k\right\|^2\leq \alpha_k^2d(1-\eta_2)^{-1}(1-\gamma)^{-1}
\end{equation}
where $d$ is the dimension of $\bbtheta$, $\eta_1<\sqrt{\eta_2}<1$, and $\gamma:=\eta_1^2/\eta_2$.
\end{lemma}

\begin{proof}
Choosing $\eta_1<1$ and defining $\gamma:=\eta_1^2/\eta_2$, it can be verified that for every $i\in\{1,\ldots,d \}$, we have
\begin{align}\label{eq.pflem4-1}
    |h_i^{k+1}|&=\left|\eta_1 h^k_i+(1-\eta_1)\nabla^k_i\right|\leq \eta_1|h_i^k|+|\nabla_i^k|\nonumber\\
    &\leq\eta_1\left(\eta_1|h_i^{k-1}|+|\nabla_i^{k-1}|\right)+|\nabla_i^k|\nonumber\\
   &\leq\sum\limits_{l=0}^k\eta_1^{k-l}|\nabla_i^l|=\sum\limits_{l=0}^k\sqrt{\gamma}^{k-l}\sqrt{\eta_2}^{k-l}|\nabla_i^l|\nonumber\\
   &\stackrel{(a)}{\leq}\left(\sum\limits_{l=0}^k\gamma^{k-l}\right)^{\frac{1}{2}}\left(\sum\limits_{l=0}^k\eta_2^{k-l}(\nabla_i^l)^2\right)^{\frac{1}{2}}\nonumber\\
   &\stackrel{(b)}{\leq}(1-\gamma)^{-\frac{1}{2}}\left(\sum\limits_{l=0}^k\eta_2^{k-l}(\nabla_i^l)^2\right)^{\frac{1}{2}}
\end{align}
where (a) follows from the Cauchy-Schwartz inequality and (b) uses $\sum\limits_{l=0}^k\gamma^{k-l}\leq \sum\limits_{k=0}^{\infty}\gamma^k\leq (1-\gamma)^{-1}$. 

For $\hat v_i^k$, first we have that $\hat v_i^1\geq(1-\eta_2)(\nabla_i^1)^2$. Then since
\begin{align*}
    \hat v_i^{k+1}\geq\eta_2\hat v_i^k+(1-\eta_2)(\nabla_i^k)^2
\end{align*}
by induction we have
\begin{equation}\label{eq.pflem4-2}
    \hat v_i^{k+1}\geq(1-\eta_2)\sum\limits_{l=0}^k\eta_2^{k-l}(\nabla_i^l)^2.
\end{equation}
Using \eqref{eq.pflem4-1} and \eqref{eq.pflem4-2}, we have
\begin{align*}
    |h_i^{k+1}|^2\leq   (1-\gamma)^{-1}\left(\sum\limits_{l=0}^k\eta_2^{k-l}(\nabla_i^l)^2\right) 
    \leq   (1-\eta_2)^{-1}(1-\gamma)^{-1}\hat v_i^{k+1}.
\end{align*}
From the update \eqref{eq.adaSCGD-3}, we have
\begin{align}
\|\bbtheta^{k+1}-\bbtheta^k\|^2 =\alpha_k^2\sum_{i=1}^d\left(\epsilon+\hat v^{k+1}_i\right)^{-1}|h_i^{k+1}|^2 \leq \alpha_k^2d(1-\eta_2)^{-1}(1-\gamma)^{-1}
\end{align}
which completes the proof.
\end{proof}

\subsubsection{Remaining steps towards Theorem \ref{theorem4}}
We are ready to prove Theorem \ref{theorem4}. 
We re-write the Lyapunov function \eqref{eq.Lyap1} as
 \begin{equation}
 \small
 	{\cal V}^k:=F(\bbtheta^k)-F(\bbtheta^*)-c_k\left\langle \nabla F(\bbtheta^{k-1}), \frac{\bbh^k}{\sqrt{\bm{\epsilon}+\hat \bbv^k}} \right\rangle + c\left\|g(\bbtheta^{k-1})-\bby^k\right\|^2
 \end{equation}
 where $c_k:=\sum\limits_{j=k}^{\infty}\eta_1^{j-k+1}\alpha_j$ and $c$ will be determined later. 
 
Using the smoothness of $F(\bbtheta^k)$ in \eqref{eq.smooth_const}, we have
{\small\begin{align}\label{eq.pfthm4-1}
 F(\bbtheta^{k+1})-F(\bbtheta^k)    &\leq \dotp{\nabla F(\bbtheta^k),\bbtheta^{k+1}-\bbtheta^k}+\frac{L}{2}\|\bbtheta^{k+1}-\bbtheta^k\|^2\nonumber\\
    &\stackrel{(a)}{=} -\alpha_k\dotp{\nabla F(\bbtheta^k), (\epsilon\mathbf{I}+\hat \bbV^{k+1})^{-\frac{1}{2}}\bbh^{k+1}}+\frac{L}{2}\|\bbtheta^{k+1}-\bbtheta^k\|^2
\end{align}}
where (a) follows from \eqref{eq.adaSCGD-3}, $\hat \bbV^{k+1}:=\diag(\hat \bbv^{k+1})$ and $(\epsilon\mathbf{I}+\hat \bbV^{k+1})^{-\frac{1}{2}}$ is understood entry-wise. 

Recalling $\bm\nabla^k:=\nabla g(\bbtheta^k;\phi^k) \nabla f(\bby^{k+1};\xi^k)$, the inner product in \eqref{eq.pfthm4-1} can be decomposed as 
\begin{align}\label{eq.pfthm4-2}
 -\dotp{\nabla F(\bbtheta^k), (\epsilon\mathbf{I}+\hat \bbV^{k+1})^{-\frac{1}{2}}\bbh^{k+1}} 
=&\underbracket{-(1-\eta_1)\dotp{\nabla F(\bbtheta^k), (\epsilon\mathbf{I}+\hat \bbV^k)^{-\frac{1}{2}}\bm\nabla^k}}_{I_1^k}\nonumber\\
&\underbracket{-\eta_1\dotp{\nabla F(\bbtheta^k), (\epsilon\mathbf{I}+\hat \bbV^k)^{-\frac{1}{2}}\bbh^k}}_{I_2^k}\nonumber\\
&\underbracket{-\dotp{\nabla F(\bbtheta^k), \left((\epsilon\mathbf{I}+\hat \bbV^{k+1})^{-\frac{1}{2}}-(\epsilon\mathbf{I}+\hat \bbV^k)^{-\frac{1}{2}}\right)\bbh^{k+1}}}_{I_3^k}. 
\end{align}

By defining $\bar{\bm\nabla}^k:=\nabla g(\bbtheta^k;\phi^k) \nabla f(g(\bbtheta^k);\xi^k)$, we have
\begin{align}\label{eq.pfthm4-3}
I_1^k=&-(1-\eta_1)\dotp{\nabla F(\bbtheta^k), (\epsilon\mathbf{I}+\hat \bbV^k)^{-\frac{1}{2}}\bar{\bm\nabla}^k}\nonumber\\
&-(1-\eta_1)\dotp{\nabla F(\bbtheta^k), (\epsilon\mathbf{I}+\hat \bbV^k)^{-\frac{1}{2}}\big(\bm\nabla^k-\bar{\bm\nabla}^k\big)}.
\end{align}
Conditioned on $\mathcal F^k$, taking expectation over $\phi^k$ and $\xi^k$ on $I_1^k$, we have
{\small\begin{align} \label{eq.pfthm4-4-0}
 \EE\left[I_1^k|\mathcal F^k\right] \stackrel{(a)}{\leq}  -(1-\eta_1)\left\|\nabla F(\bbtheta^k)\right\|^2_{(\epsilon\mathbf{I}+\hat \bbV^k)^{-\frac{1}{2}}} \!\!\!+\!(1-\eta_1)\big\|(\epsilon\mathbf{I}\!+\!\hat \bbV^k)^{-\frac{1}{4}}\!\nabla F(\bbtheta^k)\big\| \EE\!\left[\big\|(\epsilon\mathbf{I}\!+\!\hat \bbV^k)^{-\frac{1}{4}}\!\big(\bm\nabla^k\!-\!\bar{\bm\nabla}^k\big)\big\| \big |\mathcal F^k\right]
\end{align}}
where (a) uses $\EE\left[\bar{\bm\nabla}^k|\mathcal F^k\right]=\nabla F(\bbtheta^k)$. 

\blue{Expanding the second term in the RHS of \eqref{eq.pfthm4-4-0}, we have
{\small\begin{align}\label{eq.pfthm4-4}
	\EE\left[I_1^k|\mathcal F^k\right] 
	\stackrel{(b)}{\leq}& -(1-\eta_1)\left(1-\frac{\alpha_k}{4\beta_k} \right) \left\|\nabla F(\bbtheta^k)\right\|^2_{(\epsilon\mathbf{I}+\hat \bbV^k)^{-\frac{1}{2}}} + (1-\eta_1)\frac{\beta_k}{\alpha_k}\EE\left[\Big\| \bm\nabla^k-\bar{\bm\nabla}^k \Big\|^2_{(\epsilon\mathbf{I}+\hat \bbV^k)^{-\frac{1}{2}}}\big |\mathcal F^k\right]\nonumber\\
	\stackrel{(c)}{\leq}&\!-\!(1-\eta_1)\left(1-\frac{\alpha_k}{4\beta_k} \right) \left\|\nabla F(\bbtheta^k)\right\|^2_{(\epsilon\mathbf{I}+\hat \bbV^k)^{-\frac{1}{2}}} + (1-\eta_1)\frac{\beta_k}{\alpha_k}\epsilon^{-\frac{1}{2}}C_g^2L_f^2\EE\left[\Big\|g(\bbtheta^k)-\bby^{k+1}\Big\|^2\big |\mathcal F^k\right]\nonumber\\
	\stackrel{(d)}{\leq}&\!-\!(1-\eta_1)\left(1-\frac{\alpha_k}{4\beta_k} \right)(\epsilon+C_g^2C_f^2)^{-\frac{1}{2}} \left\|\nabla F(\bbtheta^k)\right\|^2 + \frac{\beta_k}{\alpha_k}\epsilon^{-\frac{1}{2}}C_g^2L_f^2\EE\left[\Big\|g(\bbtheta^k)-\bby^{k+1}\Big\|^2\big |\mathcal F^k\right]
\end{align}}
where (b) is due to the Young's inequality $ab\leq \frac{a^2}{4\beta_k}+\beta_k b^2$ with $a:=(1-\eta_1)^{\frac{1}{2}}\big\|(\epsilon\mathbf{I}\!+\!\hat \bbV^k)^{-\frac{1}{4}}\!\nabla F(\bbtheta^k)\big\|$ and $b:=(1-\eta_1)^{\frac{1}{2}}\big\|(\epsilon\mathbf{I}\!+\!\hat \bbV^k)^{-\frac{1}{4}}\!\big(\bm\nabla^k\!-\!\bar{\bm\nabla}^k\big)\big\|$; (c) uses the entrywise bound $(\epsilon+\hat{v}_i^k)^{-\frac{1}{2}}\leq  \epsilon^{-\frac{1}{2}}$ and uses Assumptions 1 and 2 to obtain
\begin{equation}
\big\|\bm\nabla^k-\bar{\bm\nabla}^k \big\|^2\leq C_g^2L_f^2\|g(\bbtheta^k)-\bby^{k+1}\|^2
\end{equation}
and (d) uses $1-\eta_1\leq 1$ and Lemma \ref{lemma3}.}

 Likewise, for $I_2^k$, we have
{\small \begin{align}\label{eq.pfthm4-5}
 \EE\left[I_2^k|\mathcal F^k\right] 	= & -\eta_1\dotp{\nabla F(\bbtheta^{k-1}), (\epsilon\mathbf{I}+\hat \bbV^k)^{-\frac{1}{2}}\bbh^k} -\eta_1\dotp{\nabla F(\bbtheta^k)-\nabla F(\bbtheta^{k-1}), (\epsilon\mathbf{I}+\hat \bbV^k)^{-\frac{1}{2}}\bbh^k}\nonumber\\
 	\stackrel{(a)}{\leq}&-\eta_1\dotp{\nabla F(\bbtheta^{k-1}), (\epsilon\mathbf{I}+\hat \bbV^k)^{-\frac{1}{2}}\bbh^k}+\eta_1L\alpha_{k-1}^{-1}\|\bbtheta^k-\bbtheta^{k-1}\|^2\nonumber\\
 \stackrel{(b)}{\leq}&\!-\eta_1\dotp{\nabla F(\bbtheta^{k-1}), (\epsilon\mathbf{I}+\hat \bbV^k)^{-\frac{1}{2}}\bbh^k}\!+\!\alpha_{k-1}\eta_1Ld(1-\eta_2)^{-1}(1-\gamma)^{-1}\nonumber\\
 \stackrel{(c)}{=}&-\eta_1(I_1^{k-1}+I_2^{k-1}+I_3^{k-1})+\alpha_{k-1}\eta_1Ld(1-\eta_2)^{-1}(1-\gamma)^{-1}
\end{align}}
where (a) follows from the $L$-smoothness of $F(\bbtheta)$ in \eqref{eq.smooth_const} implied by Assumptions 1 and 2; (b) follows from Lemma \ref{lemme4}; and (c) uses again the decomposition \eqref{eq.pfthm4-2}. 

Use $h^k_i, v^k_i, \theta^k_i, \nabla^k_i$ to denote the $i$th entry of $\bbh^k, \bbv^k, \bbtheta^k, \bm\nabla^k$. We have $|\nabla_iF(\boldsymbol{\bbtheta}^k)|\leq\|\nabla F(\boldsymbol{\bbtheta}^k)\|$, $|h_i^{k+1}|\leq\|\mathbf{h}^{k+1}\|$ and $(\epsilon+\hat{v}_i^k)^{\frac{1}{2}}\geq (\epsilon+\hat{v}_i^{k+1})^{\frac{1}{2}}$ as $\hat{v}_i^{k+1}=\max\{\cdot, \hat{v}_i^k\}\geq\hat{v}_i^k$. 

For $I_3^k$, we have
{\small\begin{align}\label{eq.pfthm4-6}
	\EE\left[I_3^k|\mathcal F^k\right] = &-\sum_{i=1}^d\nabla_i F(\bbtheta^k) \left((\epsilon+\hat v_i^{k+1})^{-\frac{1}{2}}-(\epsilon+\hat v_i^k)^{-\frac{1}{2}}\right)h_i^{k+1}\nonumber\\
	\leq&\|\nabla F(\bbtheta^k)\|\|\bbh^{k+1}\|\sum\limits_{i=1}^d\left((\epsilon+\hat v_i^k)^{-\frac{1}{2}}-(\epsilon+\hat v_i^{k+1})^{-\frac{1}{2}}\right)\nonumber\\
	\stackrel{(d)}{\leq}& C_g^2C_f^2\sum\limits_{i=1}^d\left((\epsilon+\hat v_i^k)^{-\frac{1}{2}}-(\epsilon+\hat v_i^{k+1})^{-\frac{1}{2}}\right)
\end{align}}
where (d) follows from Assumption 5 and Lemma \ref{lemma3}. 

Recalling the definition of ${\cal V}^k$ in \eqref{eq.Lyap1}, we have 
{\small\begin{align}\label{eq.pfthm4-6-2}
  {\cal V}^{k+1}-	{\cal V}^k 
 = &\,F(\bbtheta^{k+1})\!-\!F(\bbtheta^k)\!-\!c_{k+1}\left\langle \nabla F(\bbtheta^k), (\epsilon\mathbf{I}+\hat{\bbV}^{k+1})^{-\frac{1}{2}}\bbh^{k+1}\right\rangle	\nonumber\\
 &+ c\|g(\bbtheta^k)-\bby^{k+1}\|^2+c_k\left\langle \nabla F(\bbtheta^{k-1}), (\epsilon\mathbf{I}+\hat{\bbV}^k)^{-\frac{1}{2}}\bbh^k\right\rangle  - c\|g(\bbtheta^{k-1})-\bby^k\|^2\nonumber\\
 \stackrel{\eqref{eq.pfthm4-1}}{\leq}&-(\alpha_k+c_{k+1})\dotp{\nabla F(\bbtheta^k), (\epsilon\mathbf{I}+\hat \bbV^{k+1})^{-\frac{1}{2}}\bbh^{k+1}} +\frac{L}{2}\|\bbtheta^{k+1}-\bbtheta^k\|^2+ c\|g(\bbtheta^k)-\bby^{k+1}\|^2\nonumber\\
 & +c_k\left\langle \nabla F(\bbtheta^{k-1}), (\epsilon\mathbf{I}+\hat{\bbV}^k)^{-\frac{1}{2}}\bbh^k\right\rangle - c\|g(\bbtheta^{k-1})-\bby^k\|^2.
\end{align}}

Conditioned on $\mathcal F^k$, taking expectation over $\phi^k$ and $\xi^k$ on both sides of \eqref{eq.pfthm4-6-2}, we have
{\small\begin{align}\label{eq.pfthm4-6-3}
 \EE[{\cal V}^{k+1}|\mathcal F^k]-{\cal V}^k 
 \leq&(\alpha_k+c_{k+1})\EE[I_1^k+I_2^k+I_3^k\mid{\cal F}^k]+\frac{L}{2}\EE[\|\bbtheta^{k+1}-\bbtheta^k\|^2\mid{\cal F}^k]\nonumber\\
    & +c_k(I_1^{k-1}+I_2^{k-1}+I_3^{k-1})+c\EE[\|g(\bbtheta^k)-\bby^{k+1}\|^2\mid{\cal F}^k] -c\|g(\bbtheta^{k-1})-\bby^k\|^2\nonumber\\
 \stackrel{(e)}{\leq}& -(\alpha_k+c_{k+1})(1-\eta_1)\left(1-\frac{\alpha_k}{4\beta_k}\right)(\epsilon+C_g^2C_f^2)^{-\frac{1}{2}}\|\nabla F(\bbtheta^k)\|^2\nonumber\\
    &-\left((\alpha_k+c_{k+1})\eta_1-c_k\right)(I_1^{k-1}+I_2^{k-1}+I_3^{k-1})\nonumber\\
    &+(\alpha_k+c_{k+1})\alpha_{k-1}\eta_1 Ld(1-\eta_2)^{-1}(1-\gamma)^{-1}\nonumber\\
    &+(\alpha_k+c_{k+1})C_g^2C_f^2\sum\limits_{i=1}^d\left((\epsilon+\hat v_i^k)^{-\frac{1}{2}}-(\epsilon+\hat v_i^{k+1})^{-\frac{1}{2}}\right)\nonumber\\
        &+\left(c+\frac{\alpha_k+c_{k+1}}{\alpha_k}\beta_k\epsilon^{-\frac{1}{2}}C_g^2L_f^2\right)\EE[\|g(\bbtheta^k)-\bby^{k+1}\|^2\mid{\cal F}^k]\nonumber\\
        &+\frac{L}{2}\alpha_k^2(1-\eta_2)^{-1}(1-\gamma)^{-1}-c\|g(\bbtheta^{k-1})-\bby^k\|^2
\end{align}}
where (e) substitutes $\EE[I_1^k+I_2^k+I_3^k\mid{\cal F}^k]$ by \eqref{eq.pfthm4-4}-\eqref{eq.pfthm4-6} and applies Lemma \ref{lemme4}.

Selecting $\alpha_{k+1}\leq \alpha_k$ and $c_k:=\sum\limits_{j=k}^{\infty}\eta_1^{j-k+1}\alpha_j\leq (1-\eta_1)^{-1}\alpha_k$, we have
\begin{align*}
	\frac{\alpha_k+c_{k+1}}{\alpha_k }\beta_k\epsilon^{-\frac{1}{2}}C_g^2L_f^2&\leq \frac{\alpha_k+(1-\eta_1)^{-1}\alpha_{k+1}}{\alpha_k}\beta_k\epsilon^{-\frac{1}{2}}C_g^2L_f^2\nonumber\\
	&\leq \frac{\alpha_k+(1-\eta_1)^{-1}\alpha_k}{\alpha_k}\beta_k\epsilon^{-\frac{1}{2}}C_g^2L_f^2 := c \beta_k
\end{align*}
where we define $c:=(1+(1-\eta_1)^{-1})\epsilon^{-\frac{1}{2}}C_g^2L_f^2$.

Therefore, applying Lemma \ref{lemma2}, we have
{\small\begin{align}\label{eq.pfthm4-6-4}
&\Big(c+\frac{\alpha_k+c_{k+1}}{\alpha_k}\beta_k\epsilon^{-\frac{1}{2}}C_g^2L_f^2\Big)\EE[\|g(\bbtheta^k)-\bby^{k+1}\|^2\mid{\cal F}^k] -c\|g(\bbtheta^{k-1})-\bby^k\|^2\nonumber\\
		\leq &c(1+\beta_k)\EE[\|g(\bbtheta^k)-\bby^{k+1}\|^2\mid{\cal F}^k]-c\|g(\bbtheta^{k-1})-\bby^k\|^2\nonumber\\
	\leq &c\left((1+\beta_k)(1-\beta_k)^2-1\right)\|g(\bbtheta^{k-1})-\bby^k\|^2 +4c(1+\beta_k)(1-\beta_k)^2C_g^2\|\bbtheta^k-\bbtheta^{k-1}\|^2+2c(1+\beta_k)\beta_k^2V_g^2\nonumber\\
\leq &4c(1+\beta_k)(1-\beta_k)^2C_g^2\left(\alpha_{k-1}^2d(1-\eta_2)^{-1}(1-\gamma)^{-1}\right) +2c(1+\beta_k)\beta_k^2V_g^2\nonumber\\
\stackrel{(f)}{\leq} &4cC_g^2 \alpha_{k-1}^2d(1-\eta_2)^{-1}(1-\gamma)^{-1}+2c(1+\beta_k)\beta_k^2V_g^2
\end{align}}
where (f) follows from $(1+\beta_k)(1-\beta_k)^2\!=\!(1-\beta_k^2)(1-\beta_k)\!\leq\! 1$.

Selecting $c_k:=\sum\limits_{j=k}^{\infty}\eta_1^{j-k+1}\alpha_j$ implies $(\alpha_k+c_{k+1})\eta_1=c_k$. 
We thus obtain from \eqref{eq.pfthm4-6-3}-\eqref{eq.pfthm4-6-4} and $c_k \leq (1-\eta_1)^{-1} \alpha_k$ that
{\small\begin{align}\label{eq.pfthm4-7}
 \EE[{\cal V}^{k+1}|\mathcal F^k]-{\cal V}^k 
 \leq& -(\alpha_k+c_{k+1})(1-\eta_1)\Big(1-\frac{\alpha_k}{4\beta_k}\Big)(\epsilon+C_g^2C_f^2)^{-\frac{1}{2}}\|\nabla F(\bbtheta^k)\|^2\\
    &+4c C_g^2\alpha_{k-1}^2d(1-\eta_2)^{-1}(1-\gamma)^{-1}+2c\beta_k^2(1+\beta_k)V_g^2\nonumber\\
    &+(1-\eta_1)^{-1}Ld(1-\eta_2)^{-1}(1-\gamma)^{-1}\alpha_k\alpha_{k-1}\nonumber\\
    &+(\alpha_k+c_{k+1})C_g^2C_f^2\sum\limits_{i=1}^d\left((\epsilon+\hat v_i^k)^{-\frac{1}{2}}-(\epsilon+\hat v_i^{k+1})^{-\frac{1}{2}}\right) +\frac{L}{2}\alpha_k^2(1-\eta_2)^{-1}(1-\gamma)^{-1}. \nonumber
\end{align}}
 
 Defining $\tilde{\eta}:=(1-\eta_1)^{-1}(1-\eta_2)^{-1}(1-\gamma)^{-1}$ 
 and rearranging terms in \eqref{eq.pfthm4-7} and telescoping from $k=0,\cdots,K-1$, we have
{\small\begin{align*}
    &\sum\limits_{k=0}^{K-1}\alpha_k(1-\eta_1)\left(1-\frac{\alpha_k}{4\beta_k}\right)(\epsilon+C_g^2C_f^2)^{-\frac{1}{2}}\EE[\|\nabla F(\bbtheta^k)\|^2]\\
    \leq&{\cal V}^0-\EE[{\cal V}^K] + \sum\limits_{k=0}^{K-1}\left(4c C_g^2\alpha_{k-1}^2(1-\eta_1)\tilde{\eta}+2c\beta_k^2(1+\beta_k)V_g^2\right)\\
    &+\sum\limits_{k=0}^{K-1}\left(\tilde{\eta}Ld\alpha_{k-1}^2+\frac{L}{2}(1-\eta_1)\tilde{\eta}\alpha_k^2\right) +    \sum_{k=0}^{K-1}  (\alpha_k+c_{k+1})C_g^2C_f^2\sum\limits_{i=1}^d\left((\epsilon+\hat v_i^k)^{-\frac{1}{2}}-(\epsilon+\hat v_i^{k+1})^{-\frac{1}{2}}\right)\nonumber\\
  \stackrel{(g)}{\leq}&{\cal V}^0+(1-\eta_1)^{-1}\alpha_kC_gC_fd(1-\eta_1)\tilde{\eta} + \sum\limits_{k=0}^{K-1}\left(4c C_g^2\alpha_{k-1}^2(1-\eta_1)\tilde{\eta}+2c\beta_k^2(1+\beta_k)V_g^2\right)\\
    &+\sum\limits_{k=0}^{K-1}\left(\tilde{\eta}Ld\alpha_{k-1}^2+\frac{L}{2}(1-\eta_1)\tilde{\eta}\alpha_k^2\right) +(1+(1-\eta_1)^{-1})\alpha_0C_g^2C_f^2\sum_{i=1}^d\left((\epsilon+\hat v_i^0)^{-\frac{1}{2}}-(\epsilon+\hat v_i^K)^{-\frac{1}{2}}\right)
\end{align*}}
where (g) follows from $\alpha_k+c_{k+1}\leq (1+(1-\eta_1)^{-1})\alpha_k\leq (1+(1-\eta_1)^{-1})\alpha_0$ and the definition of ${\cal V}^k$ that 
{\small\begin{align*}
    \EE[{\cal V}^k] &\geq F(\bbtheta^k)-F(\bbtheta^*)+c\|g(\bbtheta^{k-1})-\bby^k\|^2\\
    &-(1-\eta_1)^{-1}\alpha_kC_gC_fd(1-\eta_2)^{-1}(1-\gamma)^{-1}.
\end{align*}}
Select $\alpha_k=2\beta_k=\alpha=\frac{1}{\sqrt{K}}$ so that $1-\frac{\alpha_k}{4\beta_k}=\frac{1}{2}$. We have that
{\small\begin{align*}
& \frac{1}{K}\sum\limits_{k=0}^{K-1}\EE[\|\nabla F(\bbtheta^k)\|^2]\leq  \frac{{\cal V}^0+ \sum\limits_{k=0}^{K-1}\left(4C_g^2(1-\eta_1)\tilde{\eta}+V_g^2\right)c\alpha^2}{K \frac{\alpha(1-\eta_1)}{2}(\epsilon+C_g^2C_f^2)^{-\frac{1}{2}}}+\frac{\sum\limits_{k=0}^{K-1}\left(\tilde{\eta}Ld+\frac{L}{2}(1-\eta_1)\tilde{\eta}\right)\alpha^2+ C_gC_fd\tilde{\eta}\alpha}{K \frac{\alpha(1-\eta_1)}{2}(\epsilon+C_g^2C_f^2)^{-\frac{1}{2}}}\nonumber\\
&\qquad\qquad\qquad\qquad\qquad\quad +\frac{(1+(1-\eta_1)^{-1})\alpha_0C_g^2C_f^2\sum_{i=1}^d(\epsilon+\hat v_i^0)^{-\frac{1}{2}}}{K \frac{\alpha(1-\eta_1)}{2}(\epsilon+C_g^2C_f^2)^{-\frac{1}{2}}}\nonumber\\
=&\frac{2(\epsilon+C_g^2C_f^2)^{\frac{1}{2}}}{(1-\eta_1)}\bigg(\!\frac{{\cal V}^0\!+\!(4C_g^2(1-\eta_1)\tilde{\eta}+V_g^2)c + (d+\frac{1}{2}(1-\eta_1))\tilde{\eta}L}{\sqrt{K}} +\!\frac{C_gC_fd\tilde{\eta}+(1+(1-\eta_1)^{-1})C_g^2C_f^2d \epsilon^{-\frac{1}{2}}}{K}\bigg)\nonumber\\
\leq &\frac{2(\epsilon+C_g^2C_f^2)^{\frac{1}{2}}}{(1-\eta_1)}\bigg( \frac{{\cal V}^0+  \left(4C_g^2\tilde{\eta}+V_g^2\right)c + 2d\tilde{\eta}L}{\sqrt{K}}+\frac{C_gC_fd\tilde{\eta}}{K} +\frac{(1+(1-\eta_1)^{-1})C_g^2C_f^2d \epsilon^{-\frac{1}{2}}}{K}\bigg)
\end{align*}}
from which the proof is complete.

\subsection{Proof of Theorem \ref{theorem-mSCSC}} 
In this section, we establish the convergence results of the multi-level SCSC, and present the corresponding analysis. 
We first prove a multi-level version of the tracking variance lemma. 
 \begin{lemma}[Tracking variance of multi-level SCSC]
\label{multi-lemma2}
If Assumptions m1-m4 hold, and $\bby_n^{k+1}$ is generated by running the multi-level SCSC iteration \eqref{eq.SCSCm} given $\bbtheta^k$, then the variance of $\bby_n^{k+1}$ satisfies 
\begin{align}\label{eq.multi-lemma2}
\EE\left[\|\bby_n^{k+1}-f_n(\bby_{n-1}^{k+1})\|^2\mid{\cal F}^k\right]\leq\, & (1-\beta_k)^2\|\bby_n^k-f_n(\bby_{n-1}^k)\|^2\nonumber\\
&+4(1-\beta_k)^2C_n^2\EE\left[\|\bby_{n-1}^k-\bby_{n-1}^{k+1}\|^2|\mathcal F^k\right]+2\beta_k^2V^2.
\end{align}
\end{lemma}
\begin{proof}
Use ${\cal F}^{k,n}$ to denote the $\sigma$-algebra generated by $\{\cdots,\bbtheta^k,\bby_1^k,\ldots,\bby_{n-1}^k\}$ From the update \eqref{eq.SCSCm}, we have that
\begin{align}\label{eq.pf-multi-lemma2-1}
\bby_n^{k+1}-f_n(\bby_{n-1}^{k+1})&=(1-\beta_k)\left(\bby_n^k-f_n(\bby_{n-1}^k)\right)+(1-\beta_k)\left(f_n(\bby_{n-1}^k)-f_n(\bby_{n-1}^{k+1})\right)\nonumber\\
&\quad+\beta_k\left(f(\bby_{n-1}^{k+1};\xi_n^k)-f_n(\bby_{n-1}^{k+1})\right)+(1-\beta_k)\left(f(\bby_{n-1}^{k+1};\xi_n^k)-f(\bby_{n-1}^k;\xi_n^k)\right)\nonumber\\
&=(1-\beta_k)(\bby_n^k-f_n(\bby_{n-1}^k))+(1-\beta_k)T_1+\beta_k T_2+(1-\beta_k)T_3
\end{align}
where we define the three terms as
\begin{align*}
    &T_1:=f_n(\bby_{n-1}^k)-f_n(\bby_{n-1}^{k+1})\\
    &T_2:=f_n(\bby_{n-1}^{k+1};\xi_n^k)-f_n(\bby_{n-1}^{k+1})\\
    &T_3:=f_n(\bby_{n-1}^{k+1};\xi_n^k)-f_n(\bby_{n-1}^{k};\xi_n^k).
\end{align*}
Conditioned on ${\cal F}^k$, taking expectation over $\phi^k$, we have
\begin{align}
\EE\left[(1-\beta_k)T_1+\beta_k T_2+(1-\beta_k)T_3|\mathcal F^k\right]=\mathbf{0}.
\end{align}
Therefore, conditioned on ${\cal F}^{k,n}:=\{\mathcal F^k, \bby_1^{k+1},\ldots,\bby_{n-1}^{k+1}\}$, taking expectation on \eqref{eq.pf-multi-lemma2-1}, we have 
\begin{align*}
\,&\EE[\|\bby_n^{k+1}-f_n(\bby_{n-1}^{k+1})\|^2\mid{\cal F}^{k,n}]\\
=\,&\EE[\|(1-\beta_k)(\bby_n^k-f_n(\bby_{n-1}^k))\|^2|\mathcal F^k]+\EE\left[\|(1-\beta_k)T_1+\beta_k T_2+(1-\beta_k)T_3\|^2\mid{\cal F}^{k,n}\right]\\
&+2\EE\left[\left\langle (1-\beta_k)(\bby_n^k-f_n(\bby_{n-1}^k)), (1-\beta_k)T_1+\beta_k T_2+(1-\beta_k)T_3\right\rangle\mid{\cal F}^{k,n}\right]\\
= &(1-\beta_k)^2\|\bby_n^k-f_n(\bby_{n-1}^k)\|^2+\EE\left[\|(1-\beta_k)T_1+\beta_k T_2+(1-\beta_k)T_3\|^2\mid{\cal F}^{k,n}\right]\\
\leq & (1-\beta_k)^2\|\bby_n^k-f_n(\bby_{n-1}^k)\|^2+2\EE\left[\|(1-\beta_k)T_1+\beta_kT_2\|^2\mid{\cal F}^{k,n}\right]+2(1-\beta_k)^2\EE\left[\|T_3\|^2\mid{\cal F}^{k,n}\right]\\
\leq & (1-\beta_k)^2\|\bby_n^k-f_n(\bby_{n-1}^k)\|^2 + 2(1-\beta_k)^2\EE[\|T_1\|^2\mid{\cal F}^{k,n}]+2\beta_k^2\EE[\|T_2\|^2\mid{\cal F}^{k,n}]\\
&+2\beta_k(1-\beta_k)\left\langle T_1,\EE[T_2\mid{\cal F}^{k,n}]\right\rangle+2(1-\beta_k)^2\EE[\|T_3\|^2\mid{\cal F}^{k,n}]\\
\leq & (1-\beta_k)^2\|\bby_n^k-f_n(\bby_{n-1}^k)\|^2+2(1-\beta_k)^2\EE\left[\|f_n(\bby_{n-1}^k)-f_n(\bby_{n-1}^{k+1})\|^2|\mathcal F^k\right]\\
&+2(1-\beta_k)^2\EE\left[\|f_n(\bby_{n-1}^k;\xi_n^k)-f_n(\bby_{n-1}^{k+1};\xi_n^k)\|^2|\mathcal F^k\right]+2\beta_k^2V^2\\
\leq & (1-\beta_k)^2\|\bby_n^k-f_n(\bby_{n-1}^k)\|^2+4(1-\beta_k)^2C_n^2\EE\left[\|\bby_{n-1}^k-\bby_{n-1}^{k+1}\|^2|\mathcal F^k\right]+2\beta_k^2V^2
\end{align*}
from which the proof is complete.
\end{proof}

Define $f^{(n)}(\bbtheta):=f_n\circ f_{n-1}\circ \cdots \circ f_1(\bbtheta)$ and the stochastic compositional gradients as
\begin{align*}
&\bm\nabla^k:=\nabla f_1(\bbtheta^k;\xi_1^k)\cdots\nabla f_{N-1}(\bby_{N-2}^{k+1};\xi_{N-1}^k)\nabla f_N(\bby_{N-1}^{k+1};\xi_N^k)\\	
&\bar{\bm\nabla}^k:= \nabla f_1(\bbtheta^k;\xi_1^k)\cdots\nabla f_{N-1}(f^{(N-2)}(\bbtheta^k);\xi_{N-1}^k)\nabla f_N(f^{(N-1)}(\bbtheta^k);\xi_N^k).
\end{align*}
Thus, taking expectation with respect to $\xi_1^k,\ldots,\xi_{N}^k$, we have
{\small
\begin{align}\label{eq.pf.multi-lemma2-3}
 \EE\left[\bm\nabla^k\mid{\cal F}^{k,N}\right]-\bar{\bm\nabla}^k 
	=&\, \nabla f_1(\bbtheta^k;\xi_1^k)\cdots\nabla f_{N-1}(\bby_{N-2}^{k+1};\xi_{N-1}^k)\nabla f_N(\bby_{N-1}^{k+1};\xi_N^k)\nonumber\\
	 &-\nabla f_1(\bbtheta^k;\xi_1^k)\cdots\nabla f_{N-1}(\bby_{N-2}^{k+1};\xi_{N-1}^k)\nabla f_N(f^{(N-1)}(\bbtheta^k);\xi_N^k)\nonumber\\
	&+\nabla f_1(\bbtheta^k;\xi_1^k)\cdots\nabla f_{N-1}(\bby_{N-2}^{k+1};\xi_{N-1}^k)\nabla f_N(f^{(N-1)}(\bbtheta^k);\xi_N^k)\nonumber\\ 
	&- \nabla f_1(\bbtheta^k;\xi_1^k)\cdots \nabla f_{N-1}(f^{(N-2)}(\bbtheta^k);\xi_{N-1}^k) \nabla f_N(f^{(N-1)}(\bbtheta^k);\xi_N^k) \nonumber\\
		    &\, \cdots \nonumber\\
	& +\nabla f_1(\bbtheta^k;\xi_1^k)\nabla f_2(\bby_1^{k+1};\xi_2^k)\cdots\nabla f_N(f^{(N-1)}(\bbtheta^k);\xi_N^k)\nonumber\\
	&-\nabla f_1(\bbtheta^k;\xi_1^k)\nabla f_2(f_1(\bbtheta^k);\xi_2^k)\cdots\nabla f_N(f^{(N-1)}(\bbtheta^k);\xi_N^k).
\end{align}
}

Since the $n$th difference term in \eqref{eq.pf.multi-lemma2-3} can be bounded by (for convenience, define $\bby_0^{k+1}=\bbtheta^k$)
{\small\begin{align}\label{eq.pf.multi-lemma2-3-2}
	&\, \Big\|\EE\Big[\nabla f_1(\bbtheta^k;\xi_1^k)\cdots\nabla f_n(\bby_{n-1}^{k+1};\xi_n^k)\cdots f_N(f^{(N-1)}(\bbtheta^k);\xi_N^k)\nonumber\\
	 &~~~-\nabla f_1(\bbtheta^k;\xi_1^k)\cdots\nabla f_n(f^{(n-1)}(\bbtheta^k);\xi_n^k)\cdots\nabla f_N(f^{(N-1)}(\bbtheta^k);\xi_N^k)\mid{\cal F}^k\Big]\Big\|\nonumber\\
	 \stackrel{(a)}{\leq} &\, \underbracket{\EE\Big[\left\|\nabla f_1(\bbtheta^k;\xi_1^k)\cdots\nabla f_{n-1}(\bby_{n-2}^{k+1};\xi_{n-1}^k)\nabla f_{n+1}(f^{(n)}(\bbtheta^k);\xi_{n+1}^k)\cdots \nabla f_N(f^{(N-1)}(\bbtheta^k);\xi_N^k) \right\|^2\mid{\cal F}^k\Big]^{\frac{1}{2}}}_{I_n^k}\nonumber\\
&\times \underbracket{\EE\Big[\left\|\nabla f_n(\bby_{n-1}^{k+1};\xi_n^k)- \nabla f_n(f^{(n-1)}(\bbtheta^k);\xi_n^k)\right\|^2\mid{\cal F}^k\Big]^{\frac{1}{2}}}_{J_n^k}
\end{align}}
where (a) uses the Cauchy-Schwartz inequality. 

For $I_n^k$, using Assumption m2, we have
{\small\begin{align*}
	I_n^k&=\EE\Big[\EE\Big[\left\|\nabla f_1(\bbtheta^k;\xi_1^k)\cdots \nabla f_N(f^{(N-1)}(\bbtheta^k);\xi_N^k) \right\|^2\mid{\cal F}^{k,N}\Big]\mid{\cal F}^k\Big]^{\frac{1}{2}}\\
	&\leq \EE\Bigg[\EE\Big[\left\|\nabla f_N(f^{(N-1)}(\bbtheta^k);\xi_N^k) \right\|^2\Big]\EE\Big[\left\|\nabla f_1(\bbtheta^k;\xi_1^k)\cdots \nabla f_{N-1}(f^{(N-2)}(\bbtheta^k);\xi_{N-1}^k) \right\|^2\mid{\cal F}^{k,N}\Big]\mid{\cal F}^k\Bigg]^{\frac{1}{2}}\\
		&\leq C_N\EE\Bigg[\EE\Big[\left\|\nabla f_1(\bbtheta^k;\xi_1^k)\cdots \nabla f_{N-1}(f^{(N-2)}(\bbtheta^k);\xi_{N-1}^k) \right\|^2\mid{\cal F}^{k,N}\Big]\mid{\cal F}^k\Bigg]^{\frac{1}{2}}\\
		&\leq C_N\EE\Bigg[\EE\Big[\left\|\nabla f_{N-1}(f^{(N-2)}(\bbtheta^k);\xi_{N-1}^k) \right\|^2\Big]\\
&~~~~~~~~~~~~~~~~\times\EE\Big[\left\|\nabla f_1(\bbtheta^k;\xi_1^k)\cdots \nabla f_{N-2}(f^{(N-3)}(\bbtheta^k);\xi_{N-2}^k) \right\|^2\mid{\cal F}^k,\bby_1^{k+1},\ldots,\bby_{N-2}^{k+1}\Big]\mid{\cal F}^k\Bigg]^{\frac{1}{2}}\\
		&\leq C_{N-1}C_N\EE\Big[\EE\Big[\left\|\nabla f_1(\bbtheta^k;\xi_1^k)\cdots \nabla f_{N-1}(f^{(N-2)}(\bbtheta^k);\xi_{N-1}^k) \right\|^2\mid{\cal F}^k,\bby_1^{k+1},\ldots,\bby_{N-2}^{k+1}\Big]\mid{\cal F}^k\Big]^{\frac{1}{2}}\\
		&\leq C_1\cdots C_{n-1}C_{n+1}\cdots C_N.
\end{align*}}

For $J_n^k$, using Assumption m1, we have
{\small\begin{align*}
	J_n^k&=\EE\Big[\left\|\nabla f_n(\bby_{n-1}^{k+1};\xi_n^k)- \nabla f_n(f^{(n-1)}(\bbtheta^k);\xi_n^k)\right\|^2\mid{\cal F}^k\Big]^{\frac{1}{2}}\\
	&\leq L_n\EE\Big[\left\|\bby_{n-1}^{k+1}-f^{(n-1)}(\bbtheta^k)\right\|\mid{\cal F}^k\Big]. 
\end{align*}}

Plugging the above two upper bounds into \eqref{eq.pf.multi-lemma2-3-2}, we have
{\small\begin{align}\label{eq.pf.multi-lemma2-3-3-1}
\left\|\EE\left[\bm\nabla^k-\bar{\bm\nabla}^k\mid {\cal F}^k\right]\right\|\nonumber
=	&\, \Big\|\sum_{n=2}^N\EE\Big[\nabla f_1(\bbtheta^k;\xi_1^k)\cdots\nabla f_n(\bby_{n-1}^{k+1};\xi_n^k)\cdots f_N(f^{(N-1)}(\bbtheta^k);\xi_N^k)\nonumber\\
	 &~~~-\nabla f_1(\bbtheta^k;\xi_1^k)\cdots\nabla f_n(f^{(n-1)}(\bbtheta^k);\xi_n^k)\cdots\nabla f_N(f^{(N-1)}(\bbtheta^k);\xi_N^k)\mid{\cal F}^k\Big]\Big\|\nonumber\\
	 \leq  & \sum_{n=2}^N C_1\cdots C_{n-1}C_{n+1}\cdots C_NL_n\EE\Big[\left\|\bby_{n-1}^{k+1}-f^{(n-1)}(\bbtheta^k)\right\|\mid{\cal F}^k\Big]\nonumber\\
	  \stackrel{(b)}{\leq} &\,\sum_{n=2}^N C_1\cdots C_{n-1}C_{n+1}\cdots C_NL_n\EE\Big[\left\|\bby_{n-1}^{k+1}-f_{n-1}(\bby_{n-2}^{k+1})\right\|\mid{\cal F}^k\Big]\nonumber\\
	 &+\sum_{n=2}^N C_1\cdots C_{n-1}C_{n+1}\cdots C_NL_n\EE\Big[\left\|f_{n-1}(\bby_{n-2}^{k+1})-f^{(n-1)}(\bbtheta^k)\right\|\mid{\cal F}^k\Big]
\end{align}}
where (b) uses the triangular inequality. 

Using the $L_{n-1}$ Lipschitz continuity of $f^{(n-1)}$, we have
{\small\begin{align}\label{eq.pf.multi-lemma2-3-3-2}
\left\|\EE\left[\bm\nabla^k-\bar{\bm\nabla}^k\mid {\cal F}^k\right]\right\|\nonumber
	 \leq &\, \sum_{n=2}^N C_1\cdots C_{n-1}C_{n+1}\cdots C_NL_n\EE\Big[\left\|\bby_{n-1}^{k+1}-f_{n-1}(\bby_{n-2}^{k+1})\right\|\mid{\cal F}^k\Big]\nonumber\\
	 &+\sum_{n=2}^NC_1\cdots C_{n-1}C_{n+1}\cdots C_NL_n L_{n-1}\EE\Big[\left\|\bby_{n-2}^{k+1}-f^{(n-2)}(\bbtheta^k)\right\|\mid{\cal F}^k\Big].
\end{align}}

Repeating the steps in \eqref{eq.pf.multi-lemma2-3-3-1} and \eqref{eq.pf.multi-lemma2-3-3-2}, we can recursively obtain
{\small\begin{align}\label{eq.pf.multi-lemma2-3-3}
\left\|\EE\left[\bm\nabla^k-\bar{\bm\nabla}^k\mid {\cal F}^k\right]\right\|\nonumber
	 \leq &\, \sum_{n=2}^N C_1\cdots C_{n-1}C_{n+1}\cdots C_NL_n\EE\Big[\left\|\bby_{n-1}^{k+1}-f_{n-1}(\bby_{n-2}^{k+1})\right\|\mid{\cal F}^k\Big]\nonumber\\
	 &+\sum_{n=2}^N C_1\cdots C_{n-1}C_{n+1}\cdots C_NL_n L_{n-1}\EE\Big[\left\|\bby_{n-2}^{k+1}-f_{n-2}(\bby_{n-3}^{k+1})\right\|\mid{\cal F}^k\Big]\nonumber\\
	&\, +\sum_{n=2}^N C_1\cdots C_{n-1}C_{n+1}\cdots C_NL_n \cdots L_{n-2}\EE\Big[\left\|\bby_{n-3}^{k+1}-f_{n-3}(\bby_{n-4}^{k+1})\right\|\mid{\cal F}^k\Big]\nonumber\\
    &\, +\cdots+\sum_{n=2}^NC_1\cdots C_{n-1}C_{n+1}\cdots C_NL_n \cdots L_2\EE\Big[\left\|\bby_1^{k+1}-f_1(\bbtheta^k)\right\|\mid{\cal F}^k\Big]\nonumber\\
	 \stackrel{(c)}{=}&\sum_{n=2}^N\sum_{m=1}^{n-1}A_{m,n}\EE\Big[\left\|\bby_m^{k+1}-f_m(\bby_{m-1}^{k+1})\right\|\mid{\cal F}^k\Big]
\end{align}}
where (c) follows by defining 
\begin{equation}
	A_{m,n}:=C_N\cdots C_{n+1}C_{n-1}\cdots C_1 L_n \cdots L_{m+1}.
\end{equation}

Therefore, using Assumption m3, we have 
{\small\begin{align}\label{eq.pf.multi-lemma2-4}
		\left\|\EE\left[\bm\nabla^k\mid {\cal F}^k\right]-\nabla F(\bbtheta^k)\right\|=&	\left\|\EE\left[\bm\nabla^k\mid {\cal F}^k\right]-\EE\left[\bar{\bm\nabla}^k\mid {\cal F}^k\right]\right\|\nonumber\\
		=&	\left\|\EE\left[\bm\nabla^k-\bar{\bm\nabla}^k\mid {\cal F}^k\right]\right\|\nonumber\\
		 \stackrel{(d)}{=} &\sum_{n=2}^N\sum_{m=1}^{n-1}A_{m,n}\EE\Big[\left\|\bby_m^{k+1}-f_m(\bby_{m-1}^{k+1})\right\|\mid{\cal F}^k\Big]\nonumber\\
		 \stackrel{(e)}{=}& \sum_{n=1}^{N-1}A_n\EE\Big[\left\|\bby_n^{k+1}-f_n(\bby_{n-1}^{k+1})\right\| \mid{\cal F}^k\Big]
\end{align}}
where (d) follows from \eqref{eq.pf.multi-lemma2-3-3} and (e) follows from $A_n:=\sum_{m=n+1}^{N-1} A_{n,m}$.

\subsection{Remaining steps towards Theorem \ref{theorem-mSCSC}}
Using the smoothness of $F(\bbtheta^k)$ in \eqref{eq.smooth_const}, we have
{\small\begin{align*}
    F(\bbtheta^{k+1})&\leq F(\bbtheta^k)+\dotp{\nabla F(\bbtheta^k),\bbtheta^{k+1}-\bbtheta^k}+\frac{L}{2}\|\bbtheta^{k+1}-\bbtheta^k\|^2\\
    &= F(\bbtheta^k)-\alpha_k\dotp{\nabla F(\bbtheta^k),\bm\nabla^k}+\frac{L}{2}\|\bbtheta^{k+1}-\bbtheta^k\|^2\\
    &= F(\bbtheta^k)-\alpha_k\|\nabla F(\bbtheta^k)\|^2+\frac{L}{2}\|\bbtheta^{k+1}-\bbtheta^k\|^2+\alpha_k\dotp{\nabla F(\bbtheta^k),\nabla F(\bbtheta^k)-\bm\nabla^k}.
\end{align*}}

Conditioned on $\mathcal F^k$, taking expectation over $\xi_1,\ldots, \xi_N$, we have
{\small\begin{align*}
    &\EE\left[F(\bbtheta^{k+1})|\mathcal F^k\right]\\
 \leq &F(\bbtheta^k)-\alpha_k\|\nabla F(\bbtheta^k)\|^2+\frac{L}{2}\EE\left[\|\bbtheta^{k+1}-\bbtheta^k\|^2|\mathcal F^k\right]+\alpha_k\left\langle\nabla F(\bbtheta^k),\EE \left[\nabla F(\bbtheta^k)-\bm\nabla^k|\mathcal F^k\right]\right\rangle\\
    \stackrel{(b)}{\leq}& F(\bbtheta^k)-\alpha_k\|\nabla F(\bbtheta^k)\|^2+\frac{L}{2}C_1^2\cdots C_N^2\alpha_k^2+\alpha_k \|\nabla F(\bbtheta^k)\|	\big\|\EE\left[\bm\nabla^k\mid{\cal F}^k\right]-\nabla F(\bbtheta^k)\big\|\\
    \stackrel{(c)}{\leq} &F(\bbtheta^k)-\alpha_k\|\nabla F(\bbtheta^k)\|^2+\frac{L}{2}C_1^2\cdots C_N^2\alpha_k^2+\alpha_k\sum_{n=1}^{N-1}A_n \|\nabla F(\bbtheta^k)\| \EE\left[\left\|\bby_n^{k+1}-f_n(\bby_{n-1}^{k+1})\right\||\mathcal F^k\right]\\
     \stackrel{(d)}{\leq}  &\,F(\bbtheta^k)\!-\!\alpha_k\!\!\left(1-\frac{\alpha_k}{4\beta_k}\sum_{n=1}^{N-1}A_n^2\right)\!\|\nabla F(\bbtheta^k)\|^2\!+\!\beta_k\!\!\sum_{n=1}^{N-1}\!\EE\left[\left\|\bby_n^{k+1}\!-\!f_n(\bby_{n-1}^{k+1})\right\|^2|\mathcal F^k\right]+\frac{L}{2}C_1^2\cdots C_N^2\alpha_k^2
\end{align*}}
where (b) uses the Cauchy-Schwartz; (c) follows from \eqref{eq.pf.multi-lemma2-4}; and (d) uses the Young's inequality. 

In addition, from the update \eqref{eq.SCSCm}, we have that
\begin{equation*}
(1-\beta_k)\left(\bby_{n-1}^{k+1}-\bby_{n-1}^k\right)\!=\!\beta_k\left(f(\bby_{n-2}^{k+1};\xi_{n-1}^k)-\bby_{n-1}^{k+1}\right)+(1-\beta_k)\left(f(\bby_{n-2}^{k+1};\xi_{n-1}^k)-f(\bby_{n-2}^k;\xi_{n-1}^k)\right).
\end{equation*}
Squaring both sides, and taking expectation conditioned on ${\cal F}^k$, we have 
{\small\begin{align}\label{eq.pf.multi-lemma2-2}
&\, \EE\left[\|\bby_{n-1}^{k+1}-\bby_{n-1}^k\|^2\mid{\cal F}^k\right]\\
\stackrel{(g)}{\leq} &\, 2\left(\frac{\beta_k}{1-\beta_k}\right)^2\!\!\EE\left[\|f_{n-1}(\bby_{n-2}^{k+1};\xi_{n-1}^k)-f_{n-1}(\bby_{n-2}^{k+1})+f_{n-1}(\bby_{n-2}^{k+1})-\bby_{n-1}^{k+1}\|^2\mid{\cal F}^k\right]\nonumber\\
&\, +2\EE\left[\left\|f_{n-1}(\bby_{n-2}^{k+1};\xi_{n-1}^k)-f_{n-1}(\bby_{n-2}^k;\xi_{n-1}^k)\right\|^2\mid{\cal F}^k\right]\nonumber\\
\leq&\, 2\left(\frac{\beta_k}{1-\beta_k}\right)^2\EE\left[\|\bby_{n-1}^{k+1}-f_{n-1}(\bby_{n-2}^{k+1})\|^2\mid{\cal F}^k\right] +2C_{n-1}^2\EE\left[\|\bby_{n-2}^{k+1}-\bby_{n-2}^k\|^2\mid{\cal F}^k\right]+2\left(\frac{\beta_k}{1-\beta_k}\right)^2V^2\nonumber
\end{align}}
where (g) follows from the Young's inequality. 

On the other hand, with the definition of Lyapunov function in \eqref{eq.Lyap2}, it follows that
{\small\begin{align}\label{eq.thm-mSCSC-1}
\!\EE[{\cal V}^{k+1}|\mathcal F^k]\leq&\, {\cal V}^k-\alpha_k\left(1-\frac{\alpha_k}{4\beta_k}\sum_{n=1}^{N-1}A_n^2\right)\|\nabla F(\bbtheta^k)\|^2+\frac{L}{2}C_1^2\cdots C_N^2\alpha_k^2-\sum_{n=1}^{N-1}\left\|\bby_n^k-f_n(\bby_{n-1}^k)\right\|^2 \nonumber\\
    &\, +(1+2\beta_k)\sum_{n=1}^{N-1}\EE\left[\left\|\bby_n^{k+1}-f_n(\bby_{n-1}^{k+1})\right\|^2|\mathcal F^k\right]  -\beta_k\sum_{n=1}^{N-1}\EE\left[\left\|\bby_n^{k+1}-f_n(\bby_{n-1}^{k+1})\right\|^2|\mathcal F^k\right]\nonumber\\  
    \stackrel{(e)}{\leq}&\, {\cal V}^k-\alpha_k\left(1-\frac{\alpha_k}{4\beta_k}\sum_{n=1}^{N-1}A_n^2\right)\|\nabla F(\bbtheta^k)\|^2+\frac{L}{2}C_1^2\cdots C_N^2\alpha_k^2+2(1+2\beta_k)\beta_k^2V^2\nonumber\\
    &\, +\left((1+2\beta_k)(1-\beta_k)^2-1\right)\sum_{n=1}^{N-1}\left\|\bby_n^k-f_n(\bby_{n-1}^k)\right\|^2\nonumber\\
    &\, +4(1+2\beta_k)(1-\beta_k)^2C_1^2\EE\left[\|\bbtheta^k-\bbtheta^{k-1}\|^2|\mathcal F^k\right]\nonumber\\
    &\, +\sum_{n=2}^{N-1}\left[4(1+2\beta_k)(1-\beta_k)^2C_n^2+\gamma_n\right]\EE\left[\|\bby_{n-1}^{k+1}-\bby_{n-1}^k\|^2|\mathcal F^k\right] \nonumber\\
     &\, -\sum_{n=2}^{N-1}\gamma_n\EE\left[\|\bby_{n-1}^{k+1}-\bby_{n-1}^k\|^2|\mathcal F^k\right] -\beta_k\sum_{n=1}^{N-1}\EE\left[\|\bby_n^{k+1}-f_n(\bby_{n-1}^{k+1})\|^2|\mathcal F^k\right]\nonumber\\
     \stackrel{(f)}{\leq}&\, {\cal V}^k-\alpha_k\left(1-\frac{\alpha_k}{4\beta_k}\sum_{n=1}^{N-1}A_n^2\right)\|\nabla F(\bbtheta^k)\|^2+\frac{L}{2}C_1^2\cdots C_N^2\alpha_k^2+2(1+2\beta_k)\beta_k^2V^2\nonumber\\
     &\, +4C_1^2\EE\left[\|\bbtheta^k-\bbtheta^{k-1}\|^2\mid{\cal F}^k\right]+\sum\limits_{n=2}^{N-1}(4C_n^2+\gamma_n)\EE\left[\|\bby_{n-1}^{k+1}-\bby_{n-1}^k\|^2\mid{\cal F}^k\right]\nonumber\\
     &\, -\sum\limits_{n=2}^{N-1}\gamma_n\EE\left[\|\bby_{n-1}^{k+1}-\bby_{n-1}^k\|^2\mid{\cal F}^k\right]-\beta_k\sum\limits_{n=1}^{N-1}\EE\left[\|\bby_n^{k+1}-f_n(\bby_{n-1}^{k+1}\|^2\mid{\cal F}^k\right]
 \end{align}}
where (e) follows from Lemma \ref{multi-lemma2}, $\bby_0^{k+1}=\bbtheta^k, \bby_0^k=\bbtheta^{k-1}$ and $\gamma_n>0$ is a fixed constant; (f) uses that $4(1+2\beta_k)(1-\beta_k)^2C_1^2\EE\left[\|\bbtheta^k-\bbtheta^{k-1}\|^2|\mathcal F^k\right]\leq 4C_1^2\EE\left[\|\bbtheta^k-\bbtheta^{k-1}\|^2|\mathcal F^k\right]$.

Plugging \eqref{eq.pf.multi-lemma2-2} into \eqref{eq.thm-mSCSC-1}, we have
{\small\begin{align}\label{eq.thm-mSCSC-2}
\!\!\! \EE[{\cal V}^{k+1}|\mathcal F^k]&\leq {\cal V}^k-\alpha_k\left(1-\frac{\alpha_k}{4\beta_k}\sum_{n=1}^{N-1}A_n^2\right)\|\nabla F(\bbtheta^k)\|^2+\frac{L}{2}C_1^2\cdots C_N^2\alpha_k^2+4 C_1^2\|\bbtheta^k-\bbtheta^{k-1}\|^2\nonumber\\
&\ \ \ +\Bigg(2(1+2\beta_k)\beta_k^2+2\left(\frac{\beta_k}{1-\beta_k}\right)^2\sum_{n=2}^{N-1}(4C_n^2+\gamma_n)\Bigg)V^2\nonumber\\
    &\ \ \ +2\left(\frac{\beta_k}{1-\beta_k}\right)^2\sum_{n=2}^{N-1}(4C_n^2+\gamma_n)\EE\left[\|\bby_{n-1}^{k+1}-f_{n-1}(\bby_{n-2}^{k+1})\|^2\mid{\cal F}^k\right]\nonumber\\
    &\ \ \ +2\sum_{n=2}^{N-1}(4C_n^2+\gamma_n)C_{n-1}^2\EE\left[\|\bby_{n-2}^{k+1}-\bby_{n-2}^k\|^2|\mathcal F^k\right]\nonumber\\
    &\ \ \ -\sum_{n=2}^{N-1}\gamma_n\EE\left[\|\bby_{n-1}^{k+1}-\bby_{n-1}^k\|^2|\mathcal F^k\right] -\beta_k\sum_{n=1}^{N-1}\EE\left[\|\bby_n^{k+1}-f_n(\bby_{n-1}^{k+1})\|^2|\mathcal F^k\right].
 \end{align}}
 
 Choose parameters $\{\gamma_n\}$ and $\{\beta_k\}$ such that
 \begin{align*}
  2(4C_n^2+\gamma_n)C_{n-1}^2\leq \gamma_{n-1}~~~{\rm and}~~~
 2\left(\frac{\beta_k}{1-\beta_k}\right)^2(4C_n^2+\gamma_n)\leq\beta_k.
 \end{align*}
For $\gamma_n$, the condition can be satisfied by choosing
\begin{equation}
\gamma_{N-1}=0,~ \gamma_{N-2}=8 C_{N-1}^2C_{N-2}^2,~ \gamma_{N-3}=16 C_{N-1}^2C_{N-2}^2C_{N-3}^2+8C_{N-2}^2C_{N-3}^2,~~~\cdots
\end{equation} 
For $\beta_k$, the condition can be satisfied by solving  following inequality that always has a solution
\begin{equation}
   \beta_k\leq \frac{\left(1-2\beta_k+(\beta_k)^2\right) C_{n-1}^2}{\gamma_{n-1}}.
\end{equation} 
 
Select $\beta_k=\beta\leq \frac{1}{2}$ and $\alpha_k=\alpha=\frac{2\beta}{\sum_{n=1}^{N-1}A_n^2}$ so that $1-\frac{\alpha_k}{4\beta_k}\sum_{n=1}^{N-1}A_n^2=\frac{1}{2}$, and define 
\begin{equation*}
    B_2:=\left(\frac{L}{2}+4C_1^2+8C_1^2C_2^2+2\gamma_2C_1^2\right)C_1^2\cdots C_N^2~~~{\rm and}~~~B_3:=4\Bigg(1+2\sum_{n=2}^{N-1}(4C_n^2+\gamma_n)\Bigg)V^2.
\end{equation*}
Plugging into \eqref{eq.thm-mSCSC-2} leads to
{\small\begin{align}\label{eq.thm-mSCSC-3}
\!\!\! \EE[{\cal V}^{k+1}]&\leq \EE[{\cal V}^k]-\frac{\alpha}{2}\EE[\|\nabla F(\bbtheta^k)\|^2]+\frac{L}{2}C_1^2\cdots C_N^2\alpha^2+2(4C_2^2+\gamma_2+2)C_1^2\EE\left[\|\bbtheta^k-\bbtheta^{k-1}\|^2\right]\nonumber\\
&\ \ \ +\Bigg(2(1+2\beta)\beta^2+2\left(\frac{\beta}{1-\beta}\right)^2\sum_{n=2}^{N-1}(4C_n^2+\gamma_n)\Bigg)V^2 \nonumber\\
& \leq \EE[{\cal V}^k]-\frac{\alpha}{2}\EE[\|\nabla F(\bbtheta^k)\|^2]+\left(\frac{L}{2}+4C_1^2+8C_1^2C_2^2+2\gamma_2\right)C_1^2\cdots C_N^2\alpha^2\nonumber\\
&\ \ \ +2\Bigg(1+2\beta+4\sum_{n=2}^{N-1}\left[4C_n^2+\gamma_n\right]\Bigg)V^2\beta^2\nonumber\\
& \leq \EE[{\cal V}^k]-\frac{\alpha}{2}\EE[\|\nabla F(\bbtheta^k)\|^2]+B_2\alpha^2+B_3 \beta^2.
\end{align}}

Choosing the stepsize as $\alpha_k=\frac{c_{\alpha}}{\sqrt{K}}$ leads to 
{\small \begin{align*}   
    \frac{\sum_{k=0}^{K-1}\EE[\|\nabla F(\bbtheta^k)\|^2]}{K}\leq
    \frac{2{\cal V}^0}{K\alpha}+2B_2\alpha+2B_3\frac{\beta^2}{\alpha}=\frac{2{\cal V}^0+2(B_2+B_3{(\sum_{n=1}^{N-1}A_n^2)^2}/{4})}{\sqrt{K}}.
\end{align*}}
This completes the proof of Theorem \ref{theorem-mSCSC}.

\end{document}